\pdfoutput=1
\documentclass[onefignum,onetabnum]{siamart190516}

\usepackage{lipsum}
\usepackage{amsfonts}
\usepackage[normalem]{ulem}
\usepackage{graphicx}
\usepackage{epstopdf}
\usepackage{pifont}
\usepackage{siunitx}
\usepackage[caption=false]{subfig}
\usepackage{pgfplotstable}
    \usetikzlibrary{
        pgfplots.colorbrewer,
    }
    \pgfplotsset{
        cycle list/.define={my marks}{
            every mark/.append style={solid,fill=\pgfkeysvalueof{/pgfplots/mark list fill}},mark=*\\
            every mark/.append style={solid,fill=\pgfkeysvalueof{/pgfplots/mark list fill}},mark=square*\\
            every mark/.append style={solid,fill=\pgfkeysvalueof{/pgfplots/mark list fill}},mark=triangle*\\
            every mark/.append style={solid,fill=\pgfkeysvalueof{/pgfplots/mark list fill}},mark=diamond*\\
        },
    }
\usepackage{algpseudocode}
\usepackage{multirow}
\usepackage{adjustbox}
\usepackage{booktabs,colortbl}
\usepackage{cprotect}
\usepackage{stmaryrd}
\usepackage{comment}

\ifpdf
  \DeclareGraphicsExtensions{.eps,.pdf,.png,.jpg}
\else
  \DeclareGraphicsExtensions{.eps}
\fi

\newsiamremark{remark}{Remark}
\newsiamremark{hypothesis}{Hypothesis}
\crefname{hypothesis}{Hypothesis}{Hypotheses}
\newsiamthm{claim}{Claim}

\headers{Preconditioner for Sparse Matrices}{H. Al Daas, P. Jolivet, F. Nataf, and P.-H. Tournier}

\title{A Robust Two-Level Schwarz Preconditioner For Sparse Matrices\thanks{Submitted to the editors \today.}}

\author{Hussam Al Daas\thanks{Scientific Computing Department, STFC, Rutherford Appleton Laboratory, Harwell Campus, Didcot, Oxfordshire, OX11 0QX, UK 
  (\email{hussam.al-daas@stfc.ac.uk}).}
  \and Pierre Jolivet\thanks{Sorbonne Universit\'e, CNRS, LIP6, 4 place Jussieu, 75252 Paris Cedex 05, France.} 
\and Fr\'ed\'eric Nataf\thanks{Sorbonne Universit\'e, CNRS, Laboratoire Jacques-Louis Lions, 4 place Jussieu, 75252 Paris Cedex 05, France. ALPINES team Inria Paris.}
\and Pierre-Henri Tournier\footnotemark[4] 
}

\usepackage{amsopn}

\newcommand{\R}{\mathbb{R}}

\newcommand{\part}[1]{\Omega_{#1}}

\ifpdf
\hypersetup{
  pdftitle={A Robust Two-Level Schwarz Preconditioner For Sparse Matrices},
  pdfauthor={H. Al Daas, P. Jolivet, F. Nataf, and P.-H Tournier}
}
\fi

\definecolor{darkpastelgreen}{rgb}{0.01, 0.75, 0.24}
\definecolor{burgundy}{rgb}{0.5, 0.0, 0.13}

\newcommand{\FN}[1]{{\color{blue}FN: #1}}

\begin{document}

\maketitle

\begin{abstract}
This paper introduces a fully algebraic two-level additive Schwarz preconditioner for general sparse large-scale matrices. The preconditioner is analyzed for symmetric positive definite (SPD) matrices. For those matrices, the coarse space is constructed based on approximating two local subspaces in each subdomain.
These subspaces are obtained by approximating a number of eigenvectors corresponding to dominant eigenvalues of two judiciously posed generalized eigenvalue problems.
The number of eigenvectors can be chosen to control the condition number.
For general sparse matrices, the coarse space is constructed by approximating the image of a local operator that can be defined from information in the coefficient matrix.
The connection between the coarse spaces for SPD and general matrices is also discussed.
Numerical experiments show the great effectiveness of the proposed preconditioners on matrices arising from a wide range of applications.
The set of matrices includes SPD, symmetric indefinite, nonsymmetric, and saddle-point matrices.
In addition, we compare the proposed preconditioners to the state-of-the-art domain decomposition preconditioners.
\end{abstract}



\section{Introduction}
In this work, we consider the linear system of equations
\begin{equation}\label{eq:Ax=b}
 A x = b,
\end{equation}
where $A\in\R^{n\times n}$ is sparse and $n\times n$ for $n\gg 1$; $b$ is the right-hand side; and $x$ is the sought solution.

This problem arises in a wide range of scientific applications and optimization.
Methods that are based on direct factorization of the sparse coefficient matrix $A$ such as sparse Cholesky, LDLT, and LU decompositions feature robust and reliable black box solvers for \cref{eq:Ax=b}, see \cite{DufER17,ScoM23}.
However, their underlying algorithms suffer from an intrinsic sequential paradigm and they require a very high amount of memory resources. Hence, sparse direct solvers are not suitable for solving large-scale linear systems of equations using emerging heterogeneous parallel computers.
Iterative methods relying on Krylov subspaces, on the other hand, require only a sparse matrix-vector multiplication ($Av$, for some $v\in\R^{n}$) per iteration which can be highly parallelizable due to the sparsity of $A$. 
The Conjugate Gradient (CG) \cite{HesS52} method, Minimal Residual (MINRES) method \cite{PaiS75}, and the Generalized Minimal Residual (GMRES) \cite{SaaS86} are widely used Krylov iterative methods for solving linear systems of equations with a coefficient matrix $A$ that is symmetric positive definite (SPD), symmetric nonsingular, general nonsingular, respectively, see \cite{Saa03} for further information on iterative methods for linear systems of equations.
The catch in these methods is that their convergence relies heavily on the properties of $A$, and they usually require a very high number of iterations to reach an approximate solution with fine accuracy.
If $A$ is SPD, the $A$-norm of the error at the $k$th iteration of CG satisfies the inequality:
\begin{equation*}\label{eq:cgbound}
 \|A^{-1}b- x_k\|_A \leq 2 \|A^{-1}b - x_0\|_A \left(\frac{\sqrt{\kappa_2\left(A\right)}-1}{\sqrt{\kappa_2 \left(A\right)}+1}\right)^k,
\end{equation*}
 where $\kappa_2\left(A\right)$ is the condition number of $A$, that is the ratio between largest and smallest eigenvalues of $A$, see \cite{Saa03}.
 Therefore, preconditioning the system \cref{eq:Ax=b}, i.e., multiplying both sides from the left by an SPD matrix $M^{-1}$ can change the spectrum of $A$ such that the ratio between the largest and smallest eigenvalues of $M^{-1}A$ is small, and hence, CG with a modified scalar product would converge fast.
Thus, if one can construct $M^{-1}$ cheaply such that the multiplication $M^{-1}v $ for $v\in\R^{n}$ is cheap and the eigenvalues of $M^{-1}A$ are clustered near 1, then, the CG runtime will be reduced.
This also applies when $A$ is symmetric and $M$ is SPD.

Although the convergence of GMRES is still a mystery for general matrices, it is known that spectral information is not enough to describe its behavior, see \cite{GrePS96,NacRT92}.
Nevertheless, it is common to employ preconditioners that are analyzed for symmetric matrices to nonsymmetric ones, see for example \cite{AldJR23}.

Preconditioning techniques vary from the simplest approach of diagonal preconditioners passing by block diagonal preconditioners and incomplete factorization up to sophisticated multilevel preconditioners based on domain decomposition (DD) methods and multigrid and even beyond by mixing some of these aforementioned, see \cite{DolHJNT22,PeaP20,Wat15,XuZ17}.

We focus in this paper on multilevel DD preconditioners, in particular, we are interested in variants of the additive Schwarz method combined with spectral coarse spaces.
These preconditioners have demonstrated, particularly during the last decade, their effectiveness in solving linear systems of equations arising from a variety of applications such as elliptic PDEs \cite{AldGJT21,AldJ22,AldJR23,GanL17,GouS21,HeiKKRW20,HeiS22,KlaKR16,KlaRR15,NatXDS11,SpiDHNPS14}, normal equations \cite{AldJS22}, indefinite Helmholtz equation \cite{BooDGCS22}, advection-diffusion equations and nonsymmetric sparse matrices \cite{AldJR23}.
Originally, spectral DD preconditioners relied essentially on the underlying PDE to set up the coarse space.
More recently, in the last three years, there have been several attempts to establish fully algebraic spectral DD preconditioners \cite{AldG19,AldJ22,AldJR23,GouS21,HeiS22}.
Despite being successful, each of these methods suffers from either requiring impractical high computational costs to set up or a lack of generality in its guaranteed robustness.
The approach proposed in \cite{AldG19} requires massive computational costs yielding impracticality issues. Nonetheless, the framework presented in \cite{AldG19} identifies a class of matrices, \emph{local symmetric positive semidefinite (SPSD) splitting} matrices and their important role in constructing robust spectral DD preconditioners.
In \cite{AldJ22}, despite the computations being local and concurrent, the complexity cost of the presented algorithm is very high, cubic with respect to the local number of unknowns per subdomain.
The robustness of the approach proposed in \cite{AldJR23} can only be guaranteed for diagonally dominant SPD matrices.
\cite{HeiS22} introduces a fully algebraic preconditioner that is guaranteed to be robust for the 2D diffusion equation.

In this paper, we present a new two-level additive Schwarz preconditioner that controls the condition number for SPD matrices.
The coarse space is obtained by solving two generalized eigenvalue problems using the Krylov--Schur method~\cite{HerRV05,Ste02}. 
The construction of the coarse space and the bound on the condition number can be established only by using local information (on each subdomain) from the coefficient matrix.
In addition, we propose an adapted version of the preconditioner to be applied to general matrices. Surprisingly, the general variant seems to be very efficient in practice after being tested on advection-dominated advection-diffusion equation (highly nonsymmetric), and symmetric indefinite saddle-point system arising from Stokes equations.

The paper is organized as follows:
\Cref{sec:background} sets the notation and gives a brief introduction to variants of algebraic additive Schwarz preconditioners.
\Cref{sec:lspsdharmonic} introduces the {lifting and harmonic extension operators} and shows the latter's connection with the local SPSD splitting. These operators will be key in defining our proposed spectral coarse spaces.
It also prepares the necessary elements required to analyze the two-level preconditioner which is then presented in \cref{sec:two-level} along with a bound on the condition number of the preconditioned matrix. Due to the large dimension of the coarse space that is proposed in \cref{sec:two-level}, another coarse space is proposed in \cref{sec:shrinking_the_coarse_space} based on an SVD-like operator truncation of the harmonic extension operator. The truncation parameter allows to adaptively control the condition number of the preconditioned matrix.
While \cref{sec:lspsdharmonic,sec:two-level,sec:shrinking_the_coarse_space} consider the SPD case, \cref{sec:nonsymmetric} introduces a generalization of the proposed coarse space to be applied to general nonsingular matrices.
Numerical experiments on a variety of sparse linear systems of equations arising from highly challenging academic examples such as the biharmonic, diffusion, advection-diffusion, elasticity, and Stokes equations are presented in \cref{sec:numerical_experiments}. They demonstrate the effectiveness and efficiency of the proposed preconditioner. Comparisons with state-of-the-art two-level domain decomposition methods show the competitiveness of the proposed method, especially being fully algebraic and not relying on any further information from the underlying problem.
Concluding remarks and future lines of research are presented in \cref{sec:conclusion}.

\section{Background}
\label{sec:background} 
\subsection{Notation}
We are given a large sparse matrix $A\in\R^{n\times n}$.
Suppose we are given a $N$ nonoverlapping partitioning of the sparsity graph of $A+A^\top$.
We refer to the $i$th subset as the $i$th nonoverlapping subdomain and it is denoted by $\Omega_{I_{i}}$.
Let $R_{I_{i}}$ denote the restriction of a vector of size $n$ to $\Omega_{I_{i}}$.
Given $\Omega_{I_{i}}$, we define the extension of $\Omega_{I_{i}}$ with $\delta$ layers to be the subset
of nodes in the sparsity graph of $A+A^\top$ that are not in $\Omega_{I_{i}}$ and reachable from $\Omega_{I_{i}}$
through a path of length smaller or equal to $\delta$. We denote this subset by $\Omega_{\Gamma_{i,1:\delta}}$.
Without loss of generality we assume that the nodes in $\Omega_{\Gamma_{i,1:\delta}}$ are ordered with respect to their distance from $\Omega_{I_{i}}$.
We refer to the subset $\Omega_{i,1:\delta} = [\Omega_{I_{i}}, \Omega_{\Gamma_{i,1:\delta}}]$ as the overlapping subdomain $i$.
We denote by $R_{i,1:\delta}$ the restriction to $\Omega_{i,1:\delta}$. The number of elements in a subset $\Theta$ is referred to as $|\Theta|$. We assume that $R_{i,1:\delta}$ orders the nodes so that the nonoverlapping nodes appear first, followed by those reachable with distance 1, then 2 and so on.
The $p\times q$ zero matrix is denoted $0_{p,q}$, and the $p\times p$ identity matrix is denoted $I_p$.
Given a vector $v\in\R^{|\Omega_{i,1:\delta}|}$ we refer to its first $|\Omega_{I_{i}}|$ components as $v_{I_{i}}$, to the following $|\Gamma_{i,1:\delta - 1}|$ components as $v_{\Gamma_{i,1:\delta-1}}$, and to the remaining $|\Gamma_{i,\delta}|$ components as $v_{\Gamma_{i,\delta}}$.

We often omit $1:\delta$ from the subscript when we refer to $\Omega_{i,1:\delta}$ or $R_{i,1:\delta}$. For example, $\Omega_i, R_i$ refer to $\Omega_{i,1:\delta}$ and $R_{i,1:\delta}$, respectively.

We define in addition the partition of unity (PoU) matrix $D_i\in\R^{|\Omega_{i,1:\delta}|\times |\Omega_{i,1:\delta}|}$ for $i=1,\ldots,N$. $D_i$ is a diagonal nonnegative matrix that satisfies the relation
\[
I_n = \sum_{i=1}^NR_i^\top D_i R_i.
\]
We consider in this paper $D_i$ that is Boolean; the first ${|\Omega_{I_{i}}|}$ diagonal values are set to one and the rest is zero.

The following example helps the reader understand the notation.
Let $A$ be given as
\[
								A = \begin{pmatrix} 1 & 2 &   &   &   &    &  \\
            																3 & 4 & 5 &   &   &    &  \\
																														& 7 & 8 & 9 &   &    &  \\
																														&   & 10& 11& 12&    &  \\
																														&   &   & 13& 14& 15 &  \\
																														&   &   &   & 16& 17 & 18\\
																														&   &   &   &   & 19 & 20\\
								\end{pmatrix}
\]
and let $N=2$, $\Omega_{I_{1}}=\{1,2,3\}$ and $\Omega_{I_{2}}=\{4,5,6,7\}$.
The restriction operators $R_{I_{1}},R_{I_{2}}$ are 
\[
								R_{I_{1}} = \begin{pmatrix} 1 & 0 & 0 & 0 & 0 &0&0\\
                      														0	& 1 & 0 & 0 & 0 &0&0 \\
																																				0	& 0 & 1 & 0 & 0 &0&0
								\end{pmatrix},
								R_{I_{2}} = \begin{pmatrix} 0 & 0  & 0  &  1 & 0  & 0  &0 \\
                      													 0	& 0  & 0  & 0		&  1 & 0  &0 \\
																																			 0	& 0  & 0  & 0		& 0  &  1 &0 \\
																																			 0	& 0  & 0  & 0		& 0  & 0  & 1
								\end{pmatrix}.
\]
The first layer reached from $\Omega_{I_{1}}$ and $\Omega_{I_{2}}$ is $\Omega_{\Gamma_{1,1:1}} = \{4\}$ and  
$\Omega_{\Gamma_{2,1:1}} = \{3\}$, respectively.
Therefore, the overlapping subdomain $\Omega_{1,1:1} = \{1,2,3,4\}$ and $\Omega_{2,1:1} = \{4,5,6,7,3\}$; the corresponding restriction operators are
\[
								R_{1,1:1} = \begin{pmatrix} 1  & 0  & 0  & 0  &  0 &0 & 0\\
                      																   0 &  1 & 0  & 0  &  0 &0 & 0\\
																																						   0 & 0  &  1 & 0  &  0 &0 & 0\\
																																									0	& 0  & 0  &  1 &  0 &0 & 0
								\end{pmatrix},
								R_{2,1:1} = \begin{pmatrix} 0 & 0  &  0 &  1 & 0  &0  & 0  \\
                          														0 & 0  &  0 & 0		&  1 &0  & 0  \\
						    																														0 & 0  &  0 & 0		& 0  &  1& 0  \\
						    																														0 & 0  &  0 & 0		& 0  &0  &  1 \\
																																								0 & 0  & 1  & 0  & 0  &0  & 0
								\end{pmatrix}.
\]
The associated Boolean PoU matrices are
\[
								D_{1} = \begin{pmatrix} 1  & 0  & 0  & 0  \\
                      										  0&  1 & 0  & 0  \\
																																  0& 0  &  1 & 0  \\
																																  0& 0  & 0  &  0 
								\end{pmatrix},
								D_{2} = \begin{pmatrix} 1  & 0  & 0  & 0  &0 \\
                          							0	&  1 & 0  & 0  &0 \\
						    																							0	& 0  &  1 & 0  &0 \\
																																	0	& 0  & 0  &  1 &0 \\
																																	0	& 0  & 0  & 0  & 0\\
								\end{pmatrix}.
\]

The first two layers reached from $\Omega_{I_{1}}$ and $\Omega_{I_{2}}$ are $\Omega_{\Gamma_{1,1:2}} = \{4,5\}$ and  
$\Omega_{\Gamma_{2,1:2}} = \{3,2\}$, respectively.
Therefore, the overlapping subdomain $\Omega_{1,1:2} = \{1,2,3,4,5\}$ and $\Omega_{2,1:2} = \{4,5,6,3,2\}$; the corresponding restriction operators are
\[
								R_{1,1:2} = \begin{pmatrix} 1  &  0 &  0 &  0 &  0 & 0 & 0 \\
                      																   0 & 1  &  0 &  0 &  0 & 0 & 0 \\
																																						   0 &  0 & 1  &  0 &  0 & 0 & 0 \\
																																									0	&  0 &  0 & 1  &  0 & 0 & 0 \\
                          														 0 &  0 &  0 & 	0	& 1  & 0 & 0 
								\end{pmatrix},
								R_{2,1:2} = \begin{pmatrix} 0 & 0  & 0  & 1  & 0  & 0  &0  \\
                          													 0	& 0  & 0  & 0		& 1  & 0  &0  \\
						    																													 0	& 0  & 0  & 0		& 0  & 1  &0  \\
						    																													 0	& 0  & 0  & 0		& 0  & 0  &1  \\
																																						  0 & 0  & 1  & 0  & 0  & 0  &0  \\
                      																  0 & 1  & 0  & 0  & 0  & 0  &0  
								\end{pmatrix}.
\]
The associated Boolean PoU matrices are
\[
								D_{1} = \begin{pmatrix}  1 & 0  & 0  & 0  & 0 \\
                      										 0 & 1  & 0  & 0  & 0 \\
																																 0 & 0  & 1  & 0  & 0 \\
																																 0 & 0  & 0  & 0  & 0 \\ 
																																 0 & 0  & 0  & 0  & 0 
								\end{pmatrix},
								D_{2} = \begin{pmatrix}  1 & 0  & 0  & 0  & 0  &0 \\
                          							0	& 1  & 0  & 0  & 0  &0 \\
						    																							0	& 0  & 1  & 0  & 0  &0 \\
																																	0	& 0  & 0  & 1  & 0  &0 \\
																																	0	& 0  & 0  & 0  & 0  &0 \\
																																	0	& 0  & 0  & 0  & 0  &0  
								\end{pmatrix}.
\]

In the following sections, the matrix $A$ is SPD unless otherwise stated.
\subsection{Subdomain-permutation}
For each subdomain $i$, let $P_i$ the permutation matrix $P_i=I_n([\Omega_{I_{i}}, \Omega_{\Gamma_{ i,1:\delta-1}},\Omega_{\Gamma_{i,\delta}}, \Omega_{\Gamma_{i,\delta+1:\infty}}], :)$.
We have the corresponding block structure,
\begin{equation}\label{eq:four_block_A}
		P_iAP_i^\top = \begin{pmatrix} A_{I_{i},I_{i}} & A_{I_{i},\Gamma_{i ,1:\delta-1}} & &\\
		                         A_{\Gamma_{i, 1:\delta-1}, I_{i}} & A_{\Gamma_{ i ,1:\delta-1}, \Gamma_{i, 1:\delta-1}} & A_{\Gamma_{i, 1:\delta-1}, \Gamma_{i, \delta}} &\\
		                          & A_{\Gamma_{i,\delta}, \Gamma_{i, 1:\delta-1},} & A_{\Gamma_{i, \delta},\Gamma_{i, \delta}} & A_{\Gamma_{i, \delta},\Gamma_{i, \delta+1:\infty}} \\ 
		                          & & A_{\Gamma_{i, \delta+1:\infty},\Gamma_{i, \delta}}& A_{\Gamma_{i, \delta+1:\infty},\Gamma_{i, \delta+1:\infty}} \end{pmatrix}.
\end{equation}
This formal representation allows us to view the matrix $A$ from a perspective where the unknowns of the subdomain under consideration are positioned first followed by the unknowns of distance one to them through the sparsity graph of $A$, which are followed by the unknowns of distance two and so on.
This is useful for defining local SPSD splitting matrices as well as the definitions of the local subspaces required to construct the spectral coarse space in subsequent sections.


\subsection{Local matrix and one-level Schwarz}
Let $A_{ii} = R_i A R_i^\top (=R_{i,1:\delta} A R_{i,1:\delta}^\top)$. Given the order in $R_i$, $A_{ii}$ takes the form
\[
A_{ii} = \begin{pmatrix} A_{I_{i},I_{i}} & A_{I_{i},\Gamma_{i ,1:\delta-1}} &\\
		                         A_{\Gamma_{i, 1:\delta-1}, I_{i}} & A_{\Gamma_{ i ,1:\delta-1}, \Gamma_{i, 1:\delta-1}} & A_{\Gamma_{i, 1:\delta-1}, \Gamma_{i, \delta}}\\
		                          & A_{\Gamma_{i,\delta}, \Gamma_{i, 1:\delta-1},} & A_{\Gamma_{i, \delta},\Gamma_{i, \delta}}
             \end{pmatrix}.
\]
Note that using the notation of subblocks of $A$, $A_{ii}=A_{\Omega_{i,1:\delta} ,\Omega_{i,1:\delta} }$.

The one-level additive Schwarz preconditioner associated with the overlapping subdomains $(\Omega_{i,1:\delta})_{1\leq i \leq N}$ is defined as:
\begin{equation}\label{eq:onelevelasm}
M_{\text{ASM}}^{-1} = \sum_{i=1}^N R_{i}^\top A_{ii}^{-1}R_i.
\end{equation}

The one-level restricted additive Schwarz preconditioner \cite{CaiS99} associated with the overlapping subdomains $(\Omega_{i,1:\delta})_{1\leq i \leq N}$ is defined as:
\begin{equation}\label{eq:onelevelras}
M_{\text{RAS}}^{-1} = \sum_{i=1}^N R_{i}^\top D_i A_{ii}^{-1}R_i,
\end{equation}
where $D_i\in\R^{|\Omega_{i,1:\delta}|\times |\Omega_{i,1:\delta}|}$ is the PoU matrix.

\section{Local SPSD splitting and local projections}
\label{sec:lspsdharmonic}
In this section, we recall the definition of the local Symmetric Positive Semi Definite (SPSD) splitting of the matrix $A$ with respect to the $i$th subdomain and present the {\it harmonic} and the {\it lifting} projections.
These two projections will define the spectral coarse space.
\subsection{Local SPSD splitting}
\begin{lemma}
\label{lemma:spsd_splitting}
Let $\widetilde{A}_i$ be defined such that
\begin{equation}\label{eq:LSPSD}
   \widetilde{A}_i := 
   P_i^\top \begin{pmatrix} A_{I_{i},I_{i}} & A_{I_{i},\Gamma_{i ,1:\delta-1}} &&\\
		                         A_{\Gamma_{i, 1:\delta-1}, I_{i}} & A_{\Gamma_{ i ,1:\delta-1}, \Gamma_{i, 1:\delta-1}} & A_{\Gamma_{i, 1:\delta-1}, \Gamma_{i, \delta}}&\\
		                          & A_{\Gamma_{i,\delta}, \Gamma_{i, 1:\delta-1},} & A_{\Gamma_{i,\delta}, \Gamma_{i, 1:\delta-1}} S_{\Gamma_{i, 1:\delta-1}}^{-1}   A_{\Gamma_{i, 1:\delta-1}, \Gamma_{i, \delta}}       &\\ &&& 0  \end{pmatrix}P_i
\end{equation}
where 
\begin{equation}\label{eq:S_G}
S_{\Gamma_{i, 1:\delta-1}} = A_{\Gamma_{ i ,1:\delta-1}, \Gamma_{i, 1:\delta-1}} - A_{\Gamma_{i, 1:\delta-1}, I_{i}}  A_{I_{i},I_{i}}^{-1} A_{I_{i},\Gamma_{i ,1:\delta-1}}.
\end{equation}
Then, $\widetilde{A}_i$ is a local SPSD splitting of $A$, i.e., $\widetilde{A}_i$ and $A-\widetilde{A}_i$ are SPSD matrices.
\end{lemma}
\begin{proof}
Consider the Schur complement w.r.t. the third block in $P_i \widetilde{A}_iP_i^\top$
$$X = A_{\Gamma_{i,\delta}, \Gamma_{i, 1:\delta-1}} S_{\Gamma_{i, 1:\delta-1}}^{-1}   A_{\Gamma_{i, 1:\delta-1}, \Gamma_{i, \delta}} - A_{\Gamma_{i,\delta}, \Omega_{i, 1:\delta-1}} A_{\Omega_{i, 1:\delta-1},\Omega_{i, 1:\delta-1}}^{-1}   A_{\Omega_{i, 1:\delta-1}, \Gamma_{i, \delta}}.$$
Since $A_{\Gamma_{i,\delta},I_{i}} = A_{I_{i},\Gamma_{i,\delta}}^\top = 0$, we have 
\begin{multline*}
A_{\Gamma_{i,\delta}, \Omega_{i, 1:\delta-1}} A_{\Omega_{i, 1:\delta-1},\Omega_{i, 1:\delta-1}}^{-1}   A_{\Omega_{i, 1:\delta-1}, \Gamma_{i, \delta}} =\\
A_{\Gamma_{i,\delta}, \Gamma_{i, 1:\delta-1}} (A_{\Gamma_{i, 1:\delta-1},\Gamma_{i, 1:\delta-1}} - A_{\Gamma_{i, 1:\delta-1}, I_{i}} A_{I_{i},I_{i}}^{-1} A_{ I_{i}, \Gamma_{i, 1:\delta-1}})^{-1}   A_{\Gamma_{i, 1:\delta-1}, \Gamma_{i, \delta}},
 \end{multline*}
 Hence $X=0$, and
since $A_{\Omega_{i, 1:\delta-1},\Omega_{i, 1:\delta-1}}$ and $A_{\Gamma_{i,\delta}, \Gamma_{i, 1:\delta-1}} S_{\Gamma_{i, 1:\delta-1}}^{-1}   A_{\Gamma_{i, 1:\delta-1}, \Gamma_{i, \delta}}$ are SPD and SPSD, respectively, we deduce that $\widetilde{A}_i$ is SPSD.

Now, consider the Schur complement of $P_iAP_i^\top -P_i \widetilde{A}_iP_i^\top$ w.r.t. the third block
 \begin{multline*}
  Y = A_{\Gamma_{i, \delta},\Gamma_{i, \delta}}	 - A_{\Gamma_{i, \delta},\Gamma_{i, 1:\delta-1}}  S_{\Gamma_{i, 1:\delta-1}}^{-1}A_{\Gamma_{i, 1:\delta-1},\Gamma_{i, \delta}}\\  - A_{\Gamma_{i, \delta},\Gamma_{i, \delta+1:\infty}}  A_{\Gamma_{i, \delta+1:\infty},\Gamma_{i, \delta+1:\infty}}^{-1}A_{\Gamma_{i, \delta+1:\infty},\Gamma_{i, \delta}}
\end{multline*}
which is nothing but the Schur complement of the matrix $P_iAP_i^\top$ w.r.t. the third block, hence $Y$ is SPD since $A$ is.
\end{proof}

The restriction of $\widetilde{A}_i$ to the overlapping subdomain $i$ is denoted by $\widetilde{A}_{ii}:=R_i \widetilde{A}_i R_i^\top$. Note that due to the locality of $\widetilde{A}_i$, we have
\begin{equation}
	\label{eq:AtildeiAtildeii}
\widetilde{A}_{i}=R_i^\top \widetilde{A}_{ii} R_i\,.	
\end{equation}

\begin{lemma}
\label{lemma:multiplicity}
Let $\widetilde{A}_i$ be the local SPSD splitting defined in \cref{eq:LSPSD} for subdomain $i$.
We have
\[
0 \leq u^\top \sum_{i=1}^N \widetilde{A}_i u \leq k_c\, u^\top A u, \qquad  \forall u \ \in \R^n
\] 
where $k_c$ is the number of colors (to be considered for the $\delta$ overlapping subdomains).
\end{lemma}
\begin{proof}
Let $G_1, \ldots, G_{k_c}$ be the partitioning of the subdomains into $k_c$ groups each gathering the subdomains with the same color.
The proof is immediate if we prove that $0 \leq u^\top \sum_{i\in G_k} \widetilde{A}_i u \leq u^\top A u$.

The last inequality can be easily argued since $ \Omega_{i,1:\delta} \cap  \Omega_{j,1:\delta}$ is empty for any $i\neq j \in G_k$, for $k = 1, \ldots, k_c$.
Let $k\in\{1,\ldots,k_c\}$. For $i\in G_k$, subtracting $\widetilde{A}_i$ from $A$ will only affect the nodes $\Omega_{i,1:\delta}$, therefore, we can safely subtract $\widetilde{A}_j$ for $j\neq i \in G_k$ from $A-\widetilde{A}_i$ and still have $A-\widetilde{A}_i-\widetilde{A}_j$ SPSD. The last inequality can be obtained by repeating the subtraction process for the rest of the elements in $G_k$. 
\end{proof}

\subsection{Local harmonic projection}
\begin{definition}[Local harmonic operator]
For subdomain $i\in\{1,\ldots,N\}$, the local harmonic operator is defined as

  \begin{align}\label{eq:HOp}
    \begin{split}
      \Pi_i:  \R^{|\Omega_{i,1:\delta}|} & \to \R^{|\Omega_{i,1:\delta}|}\\
          v=\begin{pmatrix}v_{\Omega_{i,1:\delta-1}}\\v_{\Gamma_{i,\delta}}\end{pmatrix} & \mapsto \begin{pmatrix} -A_{\Omega_{i,1:\delta-1},\Omega_{i,1:\delta-1}}^{-1} A_{\Omega_{i,1:\delta-1},\Gamma_{i,\delta}} 
          \\I_{|\Gamma_{i,\delta}|}\end{pmatrix}\begin{pmatrix}0_{|\Gamma_{i,\delta}|,|\Omega_{i,1:\delta-1}|}  & I_{|\Gamma_{i,\delta}|} \end{pmatrix} v.
    \end{split}
  \end{align}
\end{definition}
\begin{lemma}
\label{lemma:projection}
The operator $\Pi_i$ \cref{eq:HOp} is a $A_{ii}$-orthogonal projection.
That is, $\Pi_i^2 = \Pi_i$, and 
\[
(I-\Pi_i)^\top A_{ii} \Pi_i = 0.
\]
\end{lemma}
\begin{proof}
It is straightforward to see that $\Pi_i^2 = \Pi_i$. Now, let 
$u,v \in\R^{n_{|\Omega_{i,1:\delta}|}}$,
 we have 
\begin{align*}
A_{ii} \Pi_i v &= \begin{pmatrix} -A_{\Omega_{i,1:\delta-1},\Gamma_{i,\delta}} v_{\Gamma_{i,\delta}} + A_{\Omega_{i,1:\delta-1},\Gamma_{i,\delta}} v_{\Gamma_{i,\delta}}\\ -A_{\Gamma_{i,\delta},\Omega_{i,1:\delta-1}}A_{\Omega_{i,1:\delta-1},\Omega_{i,1:\delta-1}}^{-1} A_{\Omega_{i,1:\delta-1},\Gamma_{i,\delta}}  v_{\Gamma_{i,\delta}} + A_{\Gamma_{i,\delta},\Gamma_{i,\delta}}v_{\Gamma_{i,\delta}}  \end{pmatrix} \\
&=\begin{pmatrix} 0_{|\Omega_{i,1:\delta-1}|} \\ -A_{\Gamma_{i,\delta},\Omega_{i,1:\delta-1}}A_{\Omega_{i,1:\delta-1},\Omega_{i,1:\delta-1}}^{-1} A_{\Omega_{i,1:\delta-1},\Gamma_{i,\delta}}  v_{\Gamma_{i,\delta}} + A_{\Gamma_{i,\delta},\Gamma_{i,\delta}}v_{\Gamma_{i,\delta}} \end{pmatrix}.
\end{align*}
We obtain the proof by noticing that
\[
u-\Pi_i u=\begin{pmatrix} u_{\Omega_{i,1:\delta-1}}+A_{\Omega_{i,1:\delta-1},\Omega_{i,1:\delta-1}}^{-1} A_{\Omega_{i,1:\delta-1},\Gamma_{i,\delta}}  u_{\Gamma_{i,\delta}} \\ 0_{|\Gamma_{i,\delta}|} \end{pmatrix}.
\]
\end{proof}
Note that $\Pi_i v$ is harmonic for any $v\in\R^{n_{|\Omega_{i,1:\delta}|}}$. 

It is worth mentioning that there is a tight connection between the local SPSD splitting $\widetilde{A}_{ii}$ and the local harmonic operator $\Pi_i$ through the local subdomain matrix $A_{ii}$. We have
\begin{equation}
\label{eq:local_problem_local_spsd_splitting}
  (I-\Pi_i)^\top A_{ii} (I-\Pi_i) = \widetilde{A}_{ii}.
\end{equation}

\subsection{Lifting projection}
Consider the generalized eigenvalue problem
\begin{equation}\label{eq:uppergevp}
D_i A_{ii} D_i u = \lambda A_{ii} u,
\end{equation}
and let $Z_i$ be the matrix whose columns are the $A_{ii}$-orthonormal generalized eigenvectors corresponding to the
eigenvalues $\lambda > \nu$, for some prescribed $\nu > 1$.
Then the projection operator $Z_iZ_i^\top A_{ii}$ satisfies,
\begin{equation}\label{eq:Z_i}
								v^\top (I- A_{ii}Z_iZ_i^\top) D_i A_{ii} D_i (I-Z_iZ_i^\top A_{ii}) v \leq\, \nu\, v^\top A_{ii} v.
\end{equation}
Since $Z_iZ_i^\top A_{ii}$ is actually the projection on the column space of $Z_i$ parallel to the vector space spanned by all eigenvectors whose eigenvalues are lower or equal to $\nu$, eq.~\eqref{eq:Z_i} follows from~\cite[Lemma 7.7]{DolJN15}.


The following theorem characterizes the eigenvalues in \cref{eq:uppergevp}.
\begin{theorem}
The eigenvalues in the generalized eigenvalue problem \cref{eq:uppergevp}
\[
D_i A_{ii} D_i u = \lambda A_{ii} u.
\]
are composed of:
\begin{itemize}
\item 0 with multiplicity $|\Gamma_{i,1:\delta}|$
\item 1 with multiplicity dim$\left(\text{ker}\left(A_{\Gamma_{i,1:\delta},I_i}\right)\right) $ which is at least $| |I_i | - |\Gamma_{i,1:\delta}| |$
\item $\lambda > 1$
\end{itemize}

\end{theorem}
\begin{proof}
The generalized eigenvalue problem can be written as
\begin{align*}
A_{I_i,I_i} u_{I_i} &= \lambda (A_{I_i,I_i} u_{I_i} + A_{I_i,\Gamma_{i,1:\delta}} u_{\Gamma_{i,1:\delta}})\\
0 &= \lambda (A_{\Gamma_{i,1:\delta},I_i} u_{I_i} + A_{\Gamma_{i,1:\delta},\Gamma_{i,1:\delta}} u_{\Gamma_{i,1:\delta}})
\end{align*}
Notice that the pair $\left(0, \begin{pmatrix} 0_{I_i}\\ u_{\Gamma_{i,1:\delta}}\end{pmatrix}\right)$ satisfies the generalized eigenvalue problem. Hence, the multiplicity of $0$ is at least $|\Gamma_{i,1:\delta}|$, and since $A_{I_i,I_i}$ is not singular, the multiplicity of $0$ is $|\Gamma_{i,1:\delta}|$.

The pair $\left(1, \begin{pmatrix} u_{I_i} \\ 0_{\Gamma_{i,1:\delta}}\end{pmatrix}\right)$, for $u_{I_i} \in \text{ker}\left(A_{\Gamma_{i,1:\delta},I_i}\right)$, satisfies the generalized eigenvalue problem. Therefore, the minimum multiplicity of the eigenvalue 1 is the dimension of $\text{ker}\left(A_{\Gamma_{i,1:\delta},I_i}\right)$.

Now if $\lambda = 1$, we have by substitution, $A_{I_i,\Gamma_{i,1:\delta}} A_{\Gamma_{i,1:\delta},\Gamma_{i,1:\delta}} A_{\Gamma_{i,1:\delta},I_i} u_{I_i} = 0$, and since $ A_{\Gamma_{i,1:\delta},\Gamma_{i,1:\delta}} $ is SPD $A_{\Gamma_{i,1:\delta},I_i} u_{I_i} = 0$, i.e., $u_{I_i} \in \text{ker}\left(A_{\Gamma_{i,1:\delta},I_i}\right)$.
The two previous arguments show that the eigenvalue 1's multiplicity is the dimension of the kernel of  $A_{\Gamma_{i,1:\delta}, I_i} $.

The matrices defining the generalized eigenvalue problem are symmetric positive semi definite, thus $\lambda$ is non negative.  Now if $\lambda > 0$, by substitution we have
\[
A_{I_i,I_i} u_{I_i} = \lambda (A_{I_i,I_i} - A_{I_i,\Gamma_{i,1:\delta}} A_{\Gamma_{i,1:\delta},\Gamma_{i,1:\delta}}^{-1} A_{\Gamma_{i,1:\delta},I_i} ) u_{I_i}.
\]
Multiplying from the left by $u_{I_i}^\top$ we get
\[
(1-\lambda) u_{I_i}^\top A_{I_i,I_i} u_{I_i} = -\lambda u_{I_i}^\top A_{I_i,\Gamma_{i,1:\delta}} A_{\Gamma_{i,1:\delta},\Gamma_{i,1:\delta}}^{-1} A_{\Gamma_{i,1:\delta},I_i}  u_{I_i}
\]
which by using the SPSD properties of $ A_{I_i,I_i} $ and $A_{\Gamma_{i,1:\delta},\Gamma_{i,1:\delta}}$ shows that $\lambda \geq 1$ and concludes the proof.
\end{proof}
It is worth noting that the generalized eigenvalue problem~\cref{eq:uppergevp} can be interpreted as a singular value problem for $D_i$ where the Euclidean inner product is replaced by the one induced by the SPD matrix $A_{ii}$.
 
\section{Two-level theory}
\label{sec:two-level}
In this section, we set up the elements of the fictitious subspace lemma \cite[Section 7.2.1]{DolJN15} (see~\cite{Nep91} for the original paper as well as~\cite{Griebel:1995:ATA} for a modern reformulation) required to obtain spectral bounds on the two-level preconditioned matrix.

The components of the fictitious subspace lemma are the {\it decomposition spaces},  the {\it bilinear form} defined on the product of the decomposition spaces, the {\it interpolation operator}, the {\it surjectivity and continuity of the interpolation operator} and the {\it stable decomposition}.

\subsection{The coarse space and decomposition spaces}
In this section, we introduce a new spectral coarse space followed by decomposition spaces that we use to define the two-level Schwarz preconditioner. These spaces are composed of local subdomain spaces and the decomposition space associated with the spectral coarse space.

\subsubsection{The spectral coarse space}
The coarse space that we propose is defined as the space spanned by the columns of the matrix $R_0^\top$, where $R_0^\top$ is given as the horizontal concatenation of the $N$ matrices
\[
  R_i^\top D_i  \begin{pmatrix}-A_{\Omega_{i,1:\delta-1},\Omega_{i,1:\delta-1}}^{-1} A_{\Omega_{i,1:\delta-1},\Gamma_{i,\delta}} & Z_{i,\Omega_{i,1:\delta-1}} \\I_{|\Gamma_{i,\delta}|} & Z_{i,\Gamma_{i,\delta}}\end{pmatrix}.
\]
We remind the reader that  $Z_i=\begin{pmatrix} Z_{i,\Omega_{i,1:\delta-1}} \\ Z_{i,\Gamma_{i,\delta}}\end{pmatrix}$ is the matrix whose columns are the selected eigenvectors from the generalized eigenvalue problem~\cref{eq:uppergevp} such that \cref{eq:Z_i} is satisfied.
In other words, the coarse space is
\begin{equation}\label{eq:coarse_space}
 \bigoplus_{i=1}^N R_i^\top D_i  \begin{pmatrix}-A_{\Omega_{i,1:\delta-1},\Omega_{i,1:\delta-1}}^{-1} A_{\Omega_{i,1:\delta-1},\Gamma_{i,\delta}} & Z_{i,\Omega_{i,1:\delta-1}} \\I_{|\Gamma_{i,\delta}|} & Z_{i,\Gamma_{i,\delta}}\end{pmatrix}.
\end{equation}
We let $\Omega_0$ be an indexing set with $|\Omega_0|$ being equal to the dimension of the coarse space.

Given the coarse space basis matrix $R_0^\top$, we define the coarse space operator $A_{00} = R_0 A R_0^\top$ as the Galerkin projection of $A$ onto the coarse space.
We assume in what follows that $R_0^\top$ has full column rank so that $A_{00}$ is SPD.

\subsubsection{The decomposition spaces}
For each $i=0,\ldots, N$, we set $\R^{|\Omega_i|}$ as the decomposition space.
Note that for $i=1, \ldots,N$, the dimension of each decomposition space is the size of the subdomain space.

\subsection{The bilinear form}\label{sec:bilinear_form}
Now that the decomposition spaces are defined, we introduce the following bilinear form on the product space of the decomposition spaces equipped with the Euclidean scalar product:
\begin{align*}
\mathcal{B}:\prod_{i=0}^N \R^{|\Omega_i|} \times \prod_{i=0}^N \R^{|\Omega_i|} &\to \R\\
\left((u_i)_{0\leq i\leq N}, (v_i)_{0\leq i\leq N}\right) &\mapsto (v_i)_{0\leq i\leq N}^\top B (u_i)_{0\leq i\leq N} = \sum_{i=0}^N v_i^\top A_{ii} u_i,
\end{align*}
where $B: \prod_{i=0}^N \R^{|\Omega_i|} \to \prod_{i=0}^N \R^{|\Omega_i|}, \ (u_i)_{0\leq i\leq N} \mapsto (A_{ii}u_i)_{0\leq i\leq N}$.

We note that the operator ${B}$ is SPD. Furthermore, the inverse of ${B}$ can be easily obtained by replacing $A_{ii}$ with $A_{ii}^{-1}$.
\subsection{The interpolation operator}
We introduce in this section the interpolation operator $\mathcal{R}$ from the product of the decomposition spaces to the space $\R^n$ as follows
\begin{align*}
\mathcal{R}:\prod_{i=0}^N \R^{|\Omega_i|} &\to \R^n\\
(u_i)_{0\leq i\leq N} &\mapsto R_0^\top u_0 + \sum_{i=1}^N R_i^\top u_i.
\end{align*}

By using the PoU property, we have $\sum_{i=1}^N R_i^\top (D_i R_i u) = u$, for all $u\in \R^n$. Hence $\mathcal{R}$ is surjective.
In addition, $\mathcal{R}$ is continuous with respect to the norms defined by $\mathcal{B}$ and $A$ on the domain and codomain, respectively.
Indeed, for any choice of a coarse space projection matrix $R_0$, the argument of the number of colors required to color the decomposition that includes the coarse space can be used. This adds a single color to the number of colors required to color the decomposition excluding the coarse space. That is, we have the following inequality
\[
\left(\sum_{i=0}^N R_i^\top u_i \right)^\top A\left(\sum_{i=0}^N R_i^\top u_i \right) \leq (k_c+1) \sum_{i=0}^N u_i^\top A_{ii} u_i.
\]

\begin{remark}
The two-level additive Schwarz preconditioner $M_2^{-1}:=\sum_{i=0}^N R_{i}^\top A_{ii}^{-1}R_i$ can be formulated as
\[
M_2^{-1} = \mathcal{R} {B}^{-1} \mathcal{R}^\top
\]
where $B$ is defined in~\cref{sec:bilinear_form} and $\mathcal{R}^\top$ is the transpose operator of $\mathcal{R}$ which satisfies for any $u\in\R^n, (v_i)_{0\leq i\leq N}\in\prod_{i=0}^N \R^{|\Omega_i|}$
\begin{align*}
(v_i)_{0\leq i\leq N}^\top \mathcal{R}^\top u &= u^\top \mathcal{R} (v_i)_{0\leq i\leq N}\\
                                                                     &= u^\top \sum_{i=0}^NR_i^\top v_i\\
                                                                     &= \sum_{i=0}^N v_i^\top (R_iu).
                                                                     \end{align*}
That is, $\mathcal{R}^\top u = (R_i u)_{0\leq i\leq N}$.
\end{remark}

\subsection{Stable decomposition}
	\label{sec:stable_decomposition_1}
	The stable decomposition property can be stated as follows:
	
\noindent There exists a finite constant $c_l > 0$ such that $\forall u \in \R^n$, $\exists (u_i)_{0\leq i \leq N}\in\prod_{i=0}^N\R^{|\Omega_i|}$ such that $ u = \mathcal{R}  (u_i)_{0\leq i \leq N} = \sum_{i=0}^N R_i^\top u_i$ and
	\[
	  \sum_{i=0}^N u_i^\top A_{ii} u_i \leq c_l \left(\mathcal{R}  (u_i)_{0\leq i \leq N}\right)^\top A \left(\mathcal{R}  (u_i)_{0\leq i \leq N} \right) = u^\top A u.
	\]

The fictitious subspace lemma states that the largest eigenvalue of the preconditioned matrix $\mathcal{R}{B}^{-1}\mathcal{R}^\top A = M_2^{-1}A$ is bounded from above by the continuity constant of the interpolation operator. In addition, the lowest eigenvalue of $M_2^{-1}A$ is bounded from below by the inverse of the stable decomposition constant.
\begin{remark}  
Since we are working in finite dimensions, the question of the existence of the constant $c_l$ is trivial. The main question we address is to find a constant that is independent of the number of subdomains, $N$, or at least does not increase aggressively with $N$.
\end{remark}

The following auxiliary inequality, which is proved in \cite[Lemma 7.12]{DolJN15}, is useful to derive the stable decomposition inequality.
For any $R_0$, $(u_i)_{0\leq i \leq N}$, we have
\begin{equation}\label{eq:aux}
\sum_{i=0}^N u_i^\top A_{ii} u_i\, \leq\, 2 \left(\mathcal{R}(u_i)_{0\leq i \leq N}\right)^\top A \left(\mathcal{R}(u_i)_{0\leq i \leq N}\right) + \left(2k_c + 1\right) \sum_{i=1}^N u_i^\top A_{ii} u_i 
\end{equation}
	
Given $u\in\R^n$, we define the following quantities in the decomposition spaces:
\begin{itemize}
\item $u_i = D_i (I-Z_iZ_i^\top A_{ii}) (I-\Pi_i) R_i u$, for $i = 1,\ldots,N$
\item $u_0$ is defined as the stacking of the vectors
\[
\begin{pmatrix}
P_{\Gamma_{i,\delta}}^\top\\
 Z_{i}^\top  A_{ii}(I-\Pi_i)
 \end{pmatrix}  R_i u,
\]
where $P_{\Gamma_{i,\delta}} = \begin{pmatrix}
  0_{|\Gamma_{i,\delta}|,|\Omega_{i,1:\delta-1}|}  \\ I_{|\Gamma_{i,\delta}|}
  \end{pmatrix} $
such that
\[
R_0^\top u_0 = \sum_{i=1}^N R_i^\top D_i \Pi_i R_i u+ \sum_{i=1}^N R_i^\top D_i Z_i Z_i^\top A_{ii}(I-\Pi_i) R_i u
\]
\end{itemize}

Using these components from the decomposition spaces, we immediately have
\[
u = \sum_{i=0}^N R_i^\top u_i = \mathcal{R}\left(u_i\right)_{0\leq i\leq N},\ \forall u\in \R^n.
\]

As for the stability of this decomposition, the following proposition derives a stable decomposition constant.
\begin{proposition}\label{prop:stable_decomposition_1}
The decomposition introduced in~\cref{sec:stable_decomposition_1} satisfies the stable decomposition property with a constant $c_l = 2+(2k_c+1) k_c \nu$.
\end{proposition}
\begin{proof}
\begin{align*}
u_i^\top A_{ii} u_i &= \left(D_i (I-Z_iZ_i^\top A_{ii})(I-\Pi_i) R_i u\right)^\top A_{ii} \left(D_i (I-Z_iZ_i^\top A_{ii})(I-\Pi_i) R_i u\right)\\
&= \left((I-\Pi_i) R_i u\right)^\top \left((I- A_{ii}Z_iZ_i^\top) D_i A_{ii} D_i (I-Z_iZ_i^\top A_{ii}\right) \left((I-\Pi_i) R_i u\right)\\
&\stackrel{\cref{eq:Z_i}}{\leq} \nu \left((I-\Pi_i) R_i u\right)^\top A_{ii} \left((I-\Pi_i) R_i u\right)\\
&= \nu \left(R_i u\right)^\top \left((I-\Pi_i)^\top  A_{ii} (I-\Pi_i) \right) \left(R_i u\right)\\
&\stackrel{\cref{eq:local_problem_local_spsd_splitting}}{=}\nu \left(R_i u\right)^\top  \widetilde{A}_{ii} \left(R_i u\right)\\
&=\nu  u^\top  R_i^\top \widetilde{A}_{ii}R_i u\\
&\stackrel{\eqref{eq:AtildeiAtildeii}}{=} \nu u^\top  \widetilde{A}_{i}u.
\end{align*}
By summing over $i=1,\ldots,N$ and using \cref{lemma:multiplicity}, we have
\[
\sum_{i=1}^N u_i^\top A_{ii} u_i \leq k_c\, \nu\, u^\top A u.
\]
The last inequality combined with \cref{eq:aux} conclude the proof.
\end{proof}

\begin{theorem}
\label{th:upper_bound_1}
Let $M_2$ be the two-level additive Schwarz preconditioner associated with the coarse space $R_0$ defined in~\cref{eq:coarse_space}. Then, the condition number of the preconditioned matrix $M_2^{-1}A$ satisfies the inequality
\[
\kappa \left( M_2^{-1}A \right) \leq (k_c+1)  (2+(2k_c+1) k_c \nu).
\]
\end{theorem}
\begin{proof}
The proof can be obtained immediately by applying the fictitious subspace lemma with the elements introduced in~\cref{sec:two-level}.
\end{proof}
\section{Shrinking the coarse space}
	\label{sec:shrinking_the_coarse_space}
The coarse space introduced earlier can have a relatively large dimension. Indeed, each subdomain contributes a subspace of dimension $\Gamma_{i,\delta}$ at least. This results in an impractical method. To alleviate this, we judiciously select a subspace of a much smaller dimension that still provides an effective coarse space.

\subsection{The spectral coarse space and decomposition spaces}
\subsubsection{The spectral coarse space}
In this section, we will approximate the left-PoU-weighted local harmonic operator by using a form of a truncated operator decomposition to define the coarse space.
 
Recall that the operator $\Pi_i$ is given in \eqref{eq:HOp} as
\[
\Pi_i = \begin{pmatrix}-A_{\Omega_{i,1:\delta-1},\Omega_{i,1:\delta-1}}^{-1}A_{\Omega_{i,1:\delta-1},\Gamma_{i,\delta}}\\I_{|\Gamma_{i,\delta}|}\end{pmatrix}\begin{pmatrix}0_{{{|\Gamma_{i,\delta}|},|\Omega_{i,1:\delta-1}|}} & I_{{|\Gamma_{i,\delta}|}} \end{pmatrix}.
\] 
By expanding the components in $\Omega_{i,1:\delta-1}$ into components in $\Omega_{I_{i}}$ and $\Gamma_{i,1:\delta-1}$ we can write $\Pi_i$ as
\[
\Pi_i = \begin{pmatrix}
 -A_{I_{i},I_{i}}^{-1} A_{I_{i},\Gamma_{i,1:\delta-1}} S_{\Gamma_{i,1:\delta-1}}^{-1} A_{\Gamma_{i,1:\delta-1},\Gamma_{i,\delta}} \\
S_{\Gamma_{i,1:\delta-1}}^{-1} A_{\Gamma_{i,1:\delta-1},\Gamma_{i,\delta}}\\
 I_{|\Gamma_{i,\delta}|}
 \end{pmatrix}
 \begin{pmatrix}
 0_{{|\Gamma_{i,\delta}|},{|I_{i}|}}\ 0_{{{|\Gamma_{i,\delta}|},|\Gamma_{i,1:\delta-1}|}} \ I_{{|\Gamma_{i,\delta}|}}
  \end{pmatrix}.
\]

Note that since $D_i$ is Boolean it takes the value $1$ on $I_{i}$ and $0$ elsewhere, so that we have 
\[D_i\Pi_i {}= \begin{pmatrix}
 -A_{I_{i},I_{i}}^{-1} A_{I_{i},\Gamma_{i,1:\delta-1}} S_{\Gamma_{i,1:\delta-1}}^{-1} A_{\Gamma_{i,1:\delta-1},\Gamma_{i,\delta}} \\
0_{|\Gamma_{i,1:\delta-1}|,|\Gamma_{i,1:\delta-1}|}\\
 0_{|\Gamma_{i,\delta}|,|\Gamma_{i,\delta}|}
 \end{pmatrix}
 \begin{pmatrix}
 0_{{|\Gamma_{i,\delta}|},{|I_{i}|}} \ 0_{{{|\Gamma_{i,\delta}|},|\Gamma_{i,1:\delta-1}|}}\ I_{{|\Gamma_{i,\delta}|}}
  \end{pmatrix}.
\]

\begin{proposition}\label{prop:gevpsvd}
Consider the generalized eigenvalue decomposition
\begin{equation}\label{eq:gevpsvd}
\Pi_i^\top D_i A_{ii} D_i \Pi_i w = \lambda^2 A_{ii} w 
\end{equation}
and let $W_i$ be the matrix whose columns correspond to the $A_{ii}$-normalized generalized eigenvectors which form an $A_{ii}$-orthonormal basis ordered by decreasing order of the eigenvalues.
Let $W_i = [W_i^{(1)}, W_i^{(2)}]$ where 
$W_i^{(1)}$ corresponds to the set of generalized eigenvectors associated with the eigenvalues that are larger than $\tau^2$ for some prescribed number $\tau>0$.
We have
\begin{equation*}
\Pi_i^\top D_i A_{ii} D_i \Pi_i W_i^{(j)} = A_{ii} W_i^{(j)} \Lambda_j^2
\end{equation*}
where $\Lambda_j^2$ corresponds to the diagonal matrix of the eigenvalues associated with the eigenvectors $W_i^{(j)}$ for $j=1,2$.
Furthermore, since $W_i$ is $A_{ii}$-orthogonal, we have
\[
W_i^{(1)} W_i^{(1)\top}A_{ii} + W_i^{(2)} W_i^{(2)\top}A_{ii} = I.
\]
Now, let 
$\Xi_i^{(1)} = D_i \Pi_i W_i^{(1)} W_i^{(1)\top}A_{ii}$ and $\Xi_i^{(2)} = D_i \Pi_i W_i^{(2)} W_i^{(2)\top}A_{ii}$, 
we have
\begin{equation}\label{eq:svdsplitting}
D_i \Pi_i = \Xi_i^{(1)} + \Xi_i^{(2)}.
\end{equation}
Furthermore, the following holds
\begin{align}
\Xi_i^{(2)\top} A_{ii} \Xi_i^{(1)} &= 0 \label{eq:gevpsplitting1}\\
v^\top \Xi_i^{(2)\top} A_{ii} \Xi_i^{(2)} v &\leq \tau^2 v^\top A_{ii} v \label{eq:gevpsplitting2}
\end{align}
\end{proposition}
\begin{proof}
The generalized eigenvalue problem~\cref{eq:gevpsvd} is symmetric and the right-hand side matrix is SPD. Hence, the set of generalized eigenvectors, columns of $W_i$, forms an $A_{ii}$-orthonormal basis, i.e., $W_i^\top A_{ii} W_i = I$, so that $W_i^{-1}=W_i^\top A_{ii}$. That is:
\[
W_i^\top A_{ii} W_i = W_i W_i^\top A = I.
\]
Moreover, using the splitting $W_i = [W_i^{(1)}, W_i^{(2)}]$, we have
\[
[W_i^{(1)}, W_i^{(2)}] [W_i^{(1)}, W_i^{(2)}]^\top A_{ii} = W_i^{(1)} W_i^{(1)\top}A_{ii} + W_i^{(2)} W_i^{(2)\top}A_{ii} = I.
\]
Therefore,
\[
D_i \Pi_i = \Xi_i^{(1)} + \Xi_i^{(2)}\,.
\]

Using the $(\Pi_i^\top D_i A_{ii}D_i\Pi_i)$-orthogonality of the eigenvectors, we have 
\[
W_i^{(1)\top} \Pi_i^\top D_i A_{ii} D_i \Pi_i W_i^{(2)} = 0.
\]
Left and right multiplications by $A_{ii} W_i^{(1)}$ and $W_i^{(2)\top}A_{ii}$, respectively, yield the $A_{ii}$-orthogonality between $\Xi_i^{(1)}$ and $\Xi_i^{(2)}$:
\[
A_{ii} W_i^{(1)} W_i^{(1)\top} \Pi_i^\top D_i A_{ii} D_i \Pi_i W_i^{(2)} W_i^{(2)\top}A_{ii} = \Xi_i^{(1)\top} A_{ii} \Xi_i^{(2)}= 0,
\] 
which proves \cref{eq:gevpsplitting1}.
To prove \cref{eq:gevpsplitting2} we first note that since the columns of $W_i$ form a basis, any vector $v\in\R^{|\Omega_i|}$ can be written uniquely as $v = W_i^{(1)} w_1 + W_i^{(2)} w_2$, where $w_j=W_i^{(j)} A_{ii}v$ for $j=1,2$.
Therefore, we have
\begin{align*}
v^\top \Xi_i^{(2)\top} A_{ii} \Xi_i^{(2)} v &= (D_i \Pi_i W_i^{(2)} W_i^{(2)\top}A_{ii}v)^\top A_{ii} D_i \Pi_i W_i^{(2)} W_i^{(2)\top}A_{ii} v\\
&=  (D_i \Pi_i W_i^{(2)} w_2)^\top A_{ii} D_i \Pi_i W_i^{(2)} w_2\\
&=   w_2^\top W_i^{(2)\top} \Pi_i^\top D_i A_{ii} D_i \Pi_i W_i^{(2)} w_2\\
&= w_2^\top \Lambda_2^2 w_2\\
&\leq \tau^2 w_2^\top w_2\\
&= \tau^2 w_2^\top W_i^{(2)\top} A_{ii} W_i^{(2)} w_2\\
&\leq \tau^2 \left(w_2^\top W_i^{(2)\top} A_{ii} W_i^{(2)} w_2 + w_1^\top W_i^{(1)\top} A_{ii} W_i^{(1)} w_1\right)\\
&= \tau^2 v^\top A_{ii} v
\end{align*} 
\end{proof}
Note that the splitting in~\cref{eq:svdsplitting} can be seen as an SVD splitting with respect to the scalar product induced by $A_{ii}$ for the domain and codomain of $D_i \Pi_i$.

Let $\Pi_i^{(1)} = \Pi_i - D_i \Pi_i + \Xi_i^{(1)}$ and $\Pi_i^{(2)} = \Xi_i^{(2)}$, we have $\Pi_i = \Pi_i^{(1)} + \Pi_i^{(2)}$.
We define the coarse space as the space spanned by the columns of the matrix $R_0^\top$, where $R_0^\top$ is the horizontal concatenation of $R_i^\top D_i\begin{pmatrix}\Pi_iW_i^{(1)}, Z_i\end{pmatrix}$.
In other words, the coarse space is
\begin{equation}\label{eq:coarse_space_2}
\bigoplus_{i=1}^N R_i^\top D_i\begin{pmatrix}\Pi_iW_i^{(1)}, Z_i\end{pmatrix}.
\end{equation}
We let $\Omega_0$ be an indexing set with element count equal to the number of columns in $R_0^\top$.

The decomposition spaces, the bilinear form, and the interpolation operator can be defined exactly in the same way as defined earlier in~\cref{sec:two-level} and they satisfy the same properties. It remains to check the stability of the decomposition. This is carried out in the following section.

\subsection{Stable decomposition}\label{sec:stable_decomposition_2} 
Given $u\in\R^n$, we define the following quantities in the decomposition spaces:
\begin{itemize}
\item $u_i = D_i (I-Z_iZ_i^\top A_{ii}) (I-\Pi_i^{(1)}) R_i u$, for $i = 1,\ldots,N$
\item $u_0$ is defined as the vertical stacking, for $1\le i\le N,$ of the vectors
\[
\begin{pmatrix}
 W_i^{(1)\top} A_{ii}\\
 Z_{i}^\top  A_{ii}(I-\Pi_i^{(1)})
 \end{pmatrix}  R_i u.
\]
\end{itemize}
Note that $R_0^\top u_0 =\sum_{i=1}^N R_i^\top D_i \left(\Pi_i^{(1)} + Z_i Z_i^\top A_{ii}(I-\Pi_i^{(1)})\right) R_i u$.
This can be derived as follows where we exploit in the second and fifth lines the Boolean property of the PoU matrices $D_i$, that is, $D_i^2 = D_i$:
\begin{align*}
R_0^\top u_0 &= \sum_{i=1}^N R_i^\top D_i \left(\Pi_i W_i^{(1)}W_i^{(1)\top}A_{ii} + Z_i Z_i^\top A_{ii}(I-\Pi_i^{(1)})\right) R_i u\\
                      &=\sum_{i=1}^N R_i^\top D_i \left(D_i \Pi_i W_i^{(1)}W_i^{(1)\top}A_{ii} + Z_i Z_i^\top A_{ii}(I-\Pi_i^{(1)})\right) R_i u\\
                      &= \sum_{i=1}^N R_i^\top D_i \left(\Xi_i^{(1)} + Z_i Z_i^\top A_{ii}(I-\Pi_i^{(1)})\right) R_i u\\
                      &= \sum_{i=1}^N R_i^\top D_i \left(D_i \Pi_i^{(1)} + Z_i Z_i^\top A_{ii}(I-\Pi_i^{(1)})\right) R_i u\\
                      &= \sum_{i=1}^N R_i^\top D_i \left(\Pi_i^{(1)} + Z_i Z_i^\top A_{ii}(I-\Pi_i^{(1)})\right) R_i u.
\end{align*}

Using these components from the decomposition spaces, we immediately have
\[
u = \sum_{i=0}^N R_i^\top u_i = \mathcal{R}\left(u_i\right)_{0\leq i\leq N},\ \forall u\in \R^n.
\]

Regarding the stability of this decomposition, the following Proposition provides the stability constant.
\begin{proposition}
The decomposition introduced in~\cref{sec:stable_decomposition_2} satisfies the stable decomposition property with a constant 
\[c_l=\left(2+\left(2k_c+1\right)\nu\left(k_c +\lambda_*\left(2\tau + \tau^2\right)\right)\right),
\]
where $\lambda_*$ is the largest eigenvalue of the generalized eigenvalue problem 
\begin{equation}\label{eq:lambda_star}
 \sum_{i=1}^N R_i^\top A_{ii} R_i u = \lambda A u.
 \end{equation}
\end{proposition}
\begin{proof}
Let $u\in\R^n$. We have
\begin{align*}
 u^\top \widetilde{A}_i u &\stackrel{\cref{eq:local_problem_local_spsd_splitting}}{=}((I-\Pi_i) R_i u)^\top A_{ii} ((I-\Pi_i) R_i u) \\
 &= ((I-\Pi_i^{(1)} - \Pi_i^{(2)}) R_i u)^\top A_{ii} ( (I-\Pi_i^{(1)} - \Pi_i^{(2)}) R_i u)\\
 &= ( (I-\Pi_i^{(1)}) R_i u)^\top A_{ii} ( (I-\Pi_i^{(1)}) R_i u)\\
 &\quad + ( \Pi_i^{(2)} R_i u)^\top A_{ii} ( \Pi_i^{(2)} R_i u) - 2( (I-\Pi_i^{(1)}) R_i u)^\top A_{ii} ( \Pi_i^{(2)} R_i u)\\
 &= ( (I-\Pi_i^{(1)}) R_i u)^\top A_{ii} ( (I-\Pi_i^{(1)}) R_i u)\\
 &\quad +( \Pi_i^{(2)} R_i u)^\top A_{ii} ( \Pi_i^{(2)} R_i u)\\
 &\quad - 2( (I-\Pi_i^{(1)} - \Pi_i^{(2)}) R_i u)^\top A_{ii} ( \Pi_i^{(2)} R_i u)- 2( \Pi_i^{(2)} R_i u)^\top A_{ii} ( \Pi_i^{(2)} R_i u)\\
 &= ( (I-\Pi_i^{(1)}) R_i u)^\top A_{ii} ( (I-\Pi_i^{(1)}) R_i u)\\
 &\quad - ( \Pi_i^{(2)} R_i u)^\top A_{ii} ( \Pi_i^{(2)} R_i u)- 2( (I-\Pi_i) R_i u)^\top A_{ii} ( \Pi_i^{(2)} R_i u).
 \end{align*}
 Therefore, 
\begin{multline*}
 ( (I{-}\Pi_i^{(1)}) R_i u)^\top{} A_{ii}{} ( (I{-}\Pi_i^{(1)}) R_i u)= u^\top \widetilde{A}_i u\\
  \quad + 2((I-\Pi_i)R_i u)^\top  A_{ii}  (\Pi_i^{(2)} R_i u)\\
  \quad + (R_i u)^\top (\Pi_i^{(2)\top}   A_{ii}  \Pi_i^{(2)}) (R_i u).
\end{multline*}
By using Cauchy--Schwarz and \cref{prop:gevpsvd}, we have
\begin{multline*}
 ( (I{-}\Pi_i^{(1)}) R_i u)^\top{} A_{ii}{} ( (I{-}\Pi_i^{(1)}) R_i u)\leq u^\top \widetilde{A}_i u\\
 \quad + 2\sqrt{((I-\Pi_i)R_i u)^\top A_{ii}  ((I-\Pi_i)R_i u)} \sqrt{( \Pi_i^{(2)} R_i u)^\top A_{ii} (\Pi_i^{(2)}R_iu)}\\
  \quad + \tau^2(R_i u)^\top A_{ii}  (R_i u).
\end{multline*}
Since $(I-\Pi_i)$ is a $A_{ii}$-orthogonal projection and reapplying the inequality from \cref{prop:gevpsvd}, we have
\begin{multline*}
 ( (I{-}\Pi_i^{(1)}) R_i u)^\top{} A_{ii}{} ( (I{-}\Pi_i^{(1)}) R_i u)\leq u^\top \widetilde{A}_i u\\
 \quad + 2\sqrt{(R_i u)^\top A_{ii}  (R_i u)} \sqrt{\tau^2(R_i u)^\top A_{ii}  (R_i u)}\\
  \quad + \tau^2(R_i u)^\top A_{ii}  (R_i u).
\end{multline*}
Therefore,
\begin{equation*}
 ( (I{-}\Pi_i^{(1)}) R_i u)^\top A_{ii} ( (I{-}\Pi_i^{(1)}) R_i u)\leq u^\top \widetilde{A}_i u + (2\tau + \tau^2) (R_i u)^\top A_{ii}  (R_i u).
\end{equation*}
By summing over $i$, we have
\begin{align*}
 \sum_{i=1}^N ( (I-\Pi_i^{(1)}) R_i u)^\top A_{ii} ( (I-\Pi_i^{(1)}) R_i u) & \leq \sum_{i=1}^N{} u^\top \widetilde{A}_i u + (\tau^2+2 \tau) (R_i u)^\top A_{ii}  (R_i u)\\
 & \leq k_c u^\top A u + (\tau^2+2\tau) \sum_{i=1}^N (R_i u)^\top A_{ii}  (R_i u).
\end{align*}

By using the definition of $\lambda_*$, we can write
\begin{align*}
 \sum_{i=1}^N ((I-\Pi_i^{(1)}) R_i u)^\top A_{ii} ((I-\Pi_i^{(1)}) R_i u) & \leq k_c u^\top A u + (2\tau+\tau^2) \lambda_* \tau u^\top A u \\
 & \leq (k_c +\lambda_* (\tau^2+2\tau)) u^\top A u.
\end{align*}

Choosing $u_i = D_i (I-Z_iZ_i^\top A_{ii}) (I-\Pi_i^{(1)}) R_i u$, for $i=1,\ldots,N$, as in the proposed decomposition and applying~\cref{eq:Z_i}
\begin{align}\label{eq:cl}
\begin{split}
\sum_{i=1}^N u_i^\top A_{ii} u_i &\leq \nu \sum_{i=1}^N \left((I-\Pi_i^{(1)}) R_i u\right)^\top A_{ii} \left((I-\Pi_i^{(1)}) R_i u\right)\\
&\leq \nu \left(k_c +\lambda_* (\tau^2+2\tau)\right) u^\top A u.
\end{split}
\end{align}
The last inequality combined with~\cref{eq:aux} conclude the proof. 
\end{proof}

\begin{theorem}
\label{th:upper_bound_2}
Let $M_2$ be the two-level additive Schwarz preconditioner associated with the coarse space defined in~\cref{eq:coarse_space_2}. Then, the condition number of the preconditioned matrix $M_2^{-1}A$ satisfies the inequality
\[
\kappa \left( M_2^{-1}A \right) \leq \left(k_c+1\right) \left(2+\left(2k_c+1\right)\nu\left(k_c +\lambda_*\left(2\tau + \tau^2\right)\right)\right)
\]
\end{theorem}
\begin{proof}
The proof can be obtained immediately by applying the fictitious subspace lemma with the elements introduced in~\cref{sec:shrinking_the_coarse_space}.
\end{proof}
Note that all terms in the upper bound on the condition number in~\cref{th:upper_bound_2} are independent of $N$ except for the term $\lambda_*(2\tau+\tau^2)$. Indeed, the value of $\lambda_*$, the largest eigenvalue in~\cref{eq:lambda_star} might increase when $N$ increases. But, being multiplied by $(2\tau + \tau^2)$, the impact of $\lambda_*$ can be remedied by choosing small value for $\tau$.
Numerical experiments show that the eigenvalues in \cref{eq:gevpsvd} decay fast and this decay gets even faster with increasing $\delta$, the number of layers in the overlapping region.
It is then a tradeoff.

\section{Extension to nonsymmetric matrices}
\label{sec:nonsymmetric}
In the previous section, we showed how to construct a spectral coarse space by solving two generalized eigenvalue problems~\cref{eq:uppergevp,eq:gevpsvd} and selecting the eigenvectors corresponding to the dominant eigenvalues.
We noted that if we replace the Euclidean scalar product with that induced by $A_{ii}$, the selected eigenvectors in~\cref{eq:uppergevp,eq:gevpsvd} ($Z_i, W_i^{(1)}$) can be interpreted as the right singular vectors of $D_i$ and $D_i\Pi_i$ corresponding to the singular values larger than $\sqrt{\nu}$ and $\tau$, respectively.
Note also that the vectors added to the coarse space are $D_i Z_i$ and $D_i\Pi_i W_i^{(1)}$, which corresponds to the image of the {\it right singular vectors}, that is, the space spanned by the {\it left singular vectors}.


If $A$ is nonsymmetric, the local matrices may not induce a scalar product. 
One way to generalize the approach proposed previously for SPD matrices to nonsymmetric ones can be achieved by considering the singular value decomposition of $D_i$ and $D_i\Pi_i$  with respect to the Euclidean scalar product.

Since $\|D_i\|_2 = 1$ ($D_i$ is Boolean), all singular values of $D_i$ are smaller than $1$ which is smaller than $\sqrt{\nu} > 1$.
Hence, $D_i$ will not contribute anything to the coarse space in the nonsymmetric case.
The contribution from $D_i\Pi_i$ will consist of its left singular vectors associated with singular values larger than $\tau$.


\section{Numerical results}
\label{sec:numerical_experiments}
In this section, we validate the two-level method proposed earlier for solving
linear systems arising from the discretization of various PDEs. We rely on
PETSc~\cite{PETSc} as the linear algebra backend, FreeFEM~\cite{Hec12} for the
discretization, and the PCHPDDM~\cite{JolRZ21} infrastructure for benchmarking
different two-level domain decomposition preconditioners. 

We note that when the coefficient matrix is nonsymmetric, we only compute the left singular vectors associated with the largest singular values of $D_i\Pi_i$, cf.\ \cref{sec:nonsymmetric}. 
Moreover, when the coefficient matrix is SPD, we observed through a variety of numerical tests that the contribution from eigenvectors associated with the generalized eigenvalue problem \cref{eq:uppergevp} had no impact on improving the numerical effectiveness of the proposed method. Therefore, when the matrix is SPD, the PCHPDDM implementation used to generate the numerical experiments of the proposed method does not consider the generalized eigenvalue problem \cref{eq:uppergevp}.


\subsection{Impact of the number of layers}
\Cref{fig:decay_eigenvalues} demonstrates that increasing the number of overlapping layers $\delta$, speeds up the decay of eigenvalues and singular values in the generalized eigenvalue problem \cref{eq:gevpsvd} and the singular value decomposition of $D_i\Pi_i$, cf.\ \cref{sec:nonsymmetric}.
This behavior was encountered across all the numerical experiments that we carried out. 
Given a fixed threshold value $\tau$, increasing the number of overlapping layers decreases the number of contributed vectors to the coarse space at the expense of an increase in subdomain-wise local computational costs.

\pgfplotstableread{diffusion_overlap_1.dat}\diffusionA
\pgfplotstableread{diffusion_overlap_3.dat}\diffusionB
\pgfplotstableread{diffusion_overlap_5.dat}\diffusionC
\pgfplotstableread{stokes_overlap_1.dat}\stokesA
\pgfplotstableread{stokes_overlap_3.dat}\stokesB
\pgfplotstableread{stokes_overlap_5.dat}\stokesC
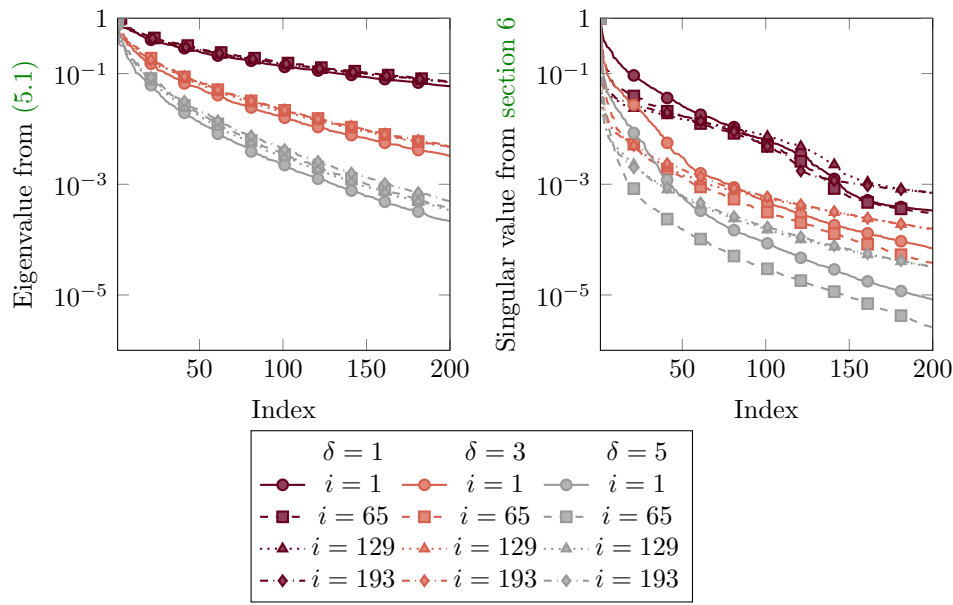
\begin{figure}[H]
\begin{tikzpicture}
\begin{semilogyaxis}[name=ax1,height=6cm,width=6cm,
    legend style={
    at={(0.4,-0.5)}, 
    anchor=west, cells={anchor=center}},
    xlabel={Index},
    clip marker paths=true,ymin=1e-6,axis on top=true,ymax=1,xmin=1,xmax=200,
    ylabel={Eigenvalue from \cref{eq:gevpsvd}},
    extra y ticks={1},
    extra y tick labels={1},
    mark list fill={.!75!white},
    cycle multiindex* list={
        my marks
            \nextlist
        [4 of]linestyles*
            \nextlist
        thick
    },
    mark repeat=20,
    legend columns=5,
    transpose legend
]
\addlegendimage{black,fill=black!50!black,mark=text,text mark={},only marks} \addlegendentry{$\delta = 1$}
\addplot+[RdGy-A] table [x expr=\coordindex+1, y expr=1/sqrt(\thisrowno{0})] {\diffusionA}; \addlegendentry{$i=1$}
\addplot+[RdGy-A] table [x expr=\coordindex+1, y expr=1/sqrt(\thisrowno{1})] {\diffusionA}; \addlegendentry{$i=65$}
\addplot+[RdGy-A] table [x expr=\coordindex+1, y expr=1/sqrt(\thisrowno{2})] {\diffusionA}; \addlegendentry{$i=129$}
\addplot+[RdGy-A] table [x expr=\coordindex+1, y expr=1/sqrt(\thisrowno{3})] {\diffusionA}; \addlegendentry{$i=193$}

\addlegendimage{black,fill=black!50!black,mark=text,text mark={},only marks} \addlegendentry{$\delta = 3$}
\addplot+[RdGy-D] table [x expr=\coordindex+1, y expr=1/sqrt(\thisrowno{0})] {\diffusionB}; \addlegendentry{$i=1$}
\addplot+[RdGy-D] table [x expr=\coordindex+1, y expr=1/sqrt(\thisrowno{1})] {\diffusionB}; \addlegendentry{$i=65$}
\addplot+[RdGy-D] table [x expr=\coordindex+1, y expr=1/sqrt(\thisrowno{2})] {\diffusionB}; \addlegendentry{$i=129$}
\addplot+[RdGy-D] table [x expr=\coordindex+1, y expr=1/sqrt(\thisrowno{3})] {\diffusionB}; \addlegendentry{$i=193$}

\addlegendimage{black,fill=black!50!black,mark=text,text mark={},only marks} \addlegendentry{$\delta = 5$}
\addplot+[RdGy-K] table [x expr=\coordindex+1, y expr=1/sqrt(\thisrowno{0})] {\diffusionC}; \addlegendentry{$i=1$}
\addplot+[RdGy-K] table [x expr=\coordindex+1, y expr=1/sqrt(\thisrowno{1})] {\diffusionC}; \addlegendentry{$i=65$}
\addplot+[RdGy-K] table [x expr=\coordindex+1, y expr=1/sqrt(\thisrowno{2})] {\diffusionC}; \addlegendentry{$i=129$}
\addplot+[RdGy-K] table [x expr=\coordindex+1, y expr=1/sqrt(\thisrowno{3})] {\diffusionC}; \addlegendentry{$i=193$}
\end{semilogyaxis}
\begin{semilogyaxis}[height=6cm,width=6cm,
at={(ax1.south east)},
xshift=2cm,
    clip marker paths=true,ymin=1e-6,axis on top=true,ymax=1,xmin=1,xmax=200,
    legend style={at={(-0.97,-0.9)},cells={anchor=west}},
    xlabel={Index},
    ylabel={Singular value from \cref{sec:nonsymmetric}},
    extra y ticks={1},
    extra y tick labels={1},
    mark list fill={.!75!white},
    cycle multiindex* list={
        my marks
            \nextlist
        [4 of]linestyles*
            \nextlist
        thick
    },
    mark repeat=20,
    legend columns=5,
    transpose legend
]
\pgfplotstablegetelem{0}{0}\of{\stokesA}
\let\zeroA\pgfplotsretval
\pgfplotstablegetelem{0}{1}\of{\stokesA}
\let\zeroB\pgfplotsretval
\pgfplotstablegetelem{0}{2}\of{\stokesA}
\let\zeroC\pgfplotsretval
\pgfplotstablegetelem{0}{3}\of{\stokesA}
\let\zeroD\pgfplotsretval
\addplot+[RdGy-A] table [x expr=\coordindex+1, y expr=\thisrowno{0}/\zeroA] {\stokesA};
\addplot+[RdGy-A] table [x expr=\coordindex+1, y expr=\thisrowno{1}/\zeroB] {\stokesA};
\addplot+[RdGy-A] table [x expr=\coordindex+1, y expr=\thisrowno{2}/\zeroC] {\stokesA};
\addplot+[RdGy-A] table [x expr=\coordindex+1, y expr=\thisrowno{3}/\zeroD] {\stokesA};

\pgfplotstablegetelem{0}{0}\of{\stokesB}
\let\zeroE\pgfplotsretval
\pgfplotstablegetelem{0}{1}\of{\stokesB}
\let\zeroF\pgfplotsretval
\pgfplotstablegetelem{0}{2}\of{\stokesB}
\let\zeroG\pgfplotsretval
\pgfplotstablegetelem{0}{3}\of{\stokesB}
\let\zeroH\pgfplotsretval
\addplot+[RdGy-D] table [x expr=\coordindex+1, y expr=\thisrowno{0}/\zeroE] {\stokesB};
\addplot+[RdGy-D] table [x expr=\coordindex+1, y expr=\thisrowno{1}/\zeroF] {\stokesB};
\addplot+[RdGy-D] table [x expr=\coordindex+1, y expr=\thisrowno{2}/\zeroG] {\stokesB};
\addplot+[RdGy-D] table [x expr=\coordindex+1, y expr=\thisrowno{3}/\zeroH] {\stokesB};

\pgfplotstablegetelem{0}{0}\of{\stokesC}
\let\zeroI\pgfplotsretval
\pgfplotstablegetelem{0}{1}\of{\stokesC}
\let\zeroJ\pgfplotsretval
\pgfplotstablegetelem{0}{2}\of{\stokesC}
\let\zeroK\pgfplotsretval
\pgfplotstablegetelem{0}{3}\of{\stokesC}
\let\zeroL\pgfplotsretval
\addplot+[RdGy-K] table [x expr=\coordindex+1, y expr=\thisrowno{0}/\zeroI] {\stokesC};
\addplot+[RdGy-K] table [x expr=\coordindex+1, y expr=\thisrowno{1}/\zeroJ] {\stokesC};
\addplot+[RdGy-K] table [x expr=\coordindex+1, y expr=\thisrowno{2}/\zeroK] {\stokesC};
\addplot+[RdGy-K] table [x expr=\coordindex+1, y expr=\thisrowno{3}/\zeroL] {\stokesC};
\end{semilogyaxis}
\end{tikzpicture} 
\caption{For a partitioning with $N=256$ subdomains of a three-dimensional diffusion (resp.\ Stokes) problem on the left-hand side (resp.\ right-hand side), the first $200$ eigenvalues $\lambda$ (resp. singular values $\sigma$) of the generalized eigenvalue problem $\Pi_i^\top D_iA_{ii}D_i\Pi u = \lambda^2 A_{ii}u$ (resp.\ $\Pi_i^\top D_i D_i\Pi_i u = \sigma^2 u$) are plotted for three values of overlapping layers $\delta = 1$, $3$, and $5$; showing values only for subdomains $i=1$, $65$, $129$, and $193$. }
\label{fig:decay_eigenvalues}
\end{figure}

\subsection{Comparison against two-level Schwarz methods}
Throughout this
comparison, the right-preconditioned GMRES~\cite{SaaS86} with no restart and a
relative tolerance set to $10^{-8}$ is used, with 256 subdomains.
The theory concerning the proposed methods holds when we combine the additive
Schwarz method~\cref{eq:onelevelras}, as the one-level correction, with an additive
coarse-space correction (second-level). As custom with domain decomposition
preconditioners, we instead use the restricted additive Schwarz method as the
one-level correction combined with a deflated variant of the
coarse-space correction, see \cite[section 3.2.1]{JolRZ21}. The setup of the coarse
operator itself remains the same, the previous remark only matters when
computing the action of the preconditioner on a vector.
The methods we compare against use the same combination as well.
Therefore, the main difference between the compared methods is the coarse space. However, the number
of overlapping layers $\delta$ may also be different.

In all that follows:
\begin{itemize}
    \item ``\ref{plot:geneo}GenEO'' (only for SPD problems) A generalized eigenvalue problem in each subdomain is solved to construct the coarse space. The involved matrices are the left-and-right-PoU-weighted subdomain matrix and the unassembled subdomain matrix (Neumann matrix)~\cite{SpiDHNPS14}. Note that this method is not algebraic since the unassembled matrix requires more information from the discretization kernel. One layer of overlap is always used.
    \item ``\ref{plot:svd}SVD'' The coarse space is assembled by finding in each subdomain $i$ the dominant subspace of $D_i\Pi_i$ which can be carried out by using the SVD, see~\cref{eq:svdsplitting}. The number of overlapping layers is $\delta$.
    \item ``\ref{plot:eps}GEVP'' (only for SPD problems) The coarse space is assembled by solving a generalized eigenvalue problem in each subdomain. The involved matrices are the subdomain matrix and the left-and-right-PoU-weighted subdomain matrix, see~\cref{eq:gevpsvd}. The number of overlapping layers is $\delta$.
    \item ``\ref{plot:block}Block splitting'' A generalized eigenvalue problem is solved in each subdomain to construct a coarse space. One of the matrices involved is the left-and-right-PoU-weighted subdomain matrix. The other matrix is obtained algebraically based on a heuristic (which is guaranteed for Hermitian diagonally dominant matrices), see~\cite{AldJR23}. One layer of overlap is used.
\end{itemize}

The following problems are solved.
\begin{itemize}
    \item \cref{fig:bilaplace}, constant-coefficient fourth-order two-dimensional bilaplacian discretized with Morley finite elements~\cite{peric1991morley};
    \item \cref{fig:diffusion-heterogeneous}, heterogeneous-coefficient three-dimensional Poisson equation discretized with piecewise-linear Lagrange finite elements;
    \item \cref{fig:elasticity-2d}, constant-coefficient two-dimensional linear elasticity discretized with piecewise-linear Lagrange finite elements;
    \item \cref{fig:elasticity-3d}, heterogeneous-coefficient three-dimensional linear elasticity discretized with piecewise-linear Lagrange finite elements (for the coefficient distribution, see Figure 6.2 of~\cite{AldGJT21});
    \item \cref{fig:stokes-2d}, constant-coefficient two-dimensional lid-driven cavity problem discretized with lowest-order Taylor--Hood finite elements;
    \item \cref{fig:stokes-3d}, constant-coefficient three-dimensional lid-driven cavity problem discretized with lowest-order Taylor--Hood finite elements;
    \item \cref{fig:advection-2d}, heterogeneous-coefficient two-dimensional SUPG-stabilized advection equation discretized with piecewise-linear Lagrange finite elements (for the coefficient distribution, see Figure 2 (b) of~\cite{AldJR23});
    \item \cref{fig:advection-3d}, heterogeneous-coefficient three-dimensional SUPG-stabilized advection equation discretized with piecewise-linear Lagrange finite elements (for the coefficient distribution, see Figure 2 (b) of~\cite{AldJR23});
\end{itemize}
The parameters (overlap and threshold for selecting eigenvectors or singular vectors) for the various preconditioners are displayed on the bottom right-hand side table for all previously listed figures.
Overall, while convergence is under control for all these physics, one can notice that the grid and operator complexities (GC and OC), of the proposed two-level preconditioner are higher than previously developed preconditioners. We recall that, with a two-level method, $\text{GC} = 1 + \frac{n_C}{n}$ and $\text{OC} = 1 + \frac{\text{nnz}_C}{\text{nnz}_A}$ where $\text{nnz}_C$ (resp.\ $\text{nnz}_A$) is the number of nonzero entries in the coarse (resp.\ fine) operator.

\pgfplotstableread{figures/bilaplace_2d.dat}\loadedconvergence
\pgfplotstableread{figures/bilaplace_2d_op.dat}\sizetable
\pgfplotstableread{figures/bilaplace_2d_param.dat}\paramtable
\begin{figure}
\pgfplotstablegetelem{0}{svd}\of{\loadedconvergence}
\pgfkeys{/pgf/fpu}
    \pgfplotstablecreatecol[create col/expr={\thisrow{geneo}/\pgfplotsretval}]{geneo-rel}\loadedconvergence
    \pgfplotstablecreatecol[create col/expr={\thisrow{svd}/\pgfplotsretval}]{svd-rel}\loadedconvergence
    \pgfplotstablecreatecol[create col/expr={\thisrow{eps}/\pgfplotsretval}]{eps-rel}\loadedconvergence
    \pgfmathparse{multiply(1e-8,1)}
    \let\myextratick=\pgfmathresult
    \pgfkeys{/pgf/fpu=false}
    \begin{minipage}{0.6\textwidth}
\begin{tikzpicture}
\begin{semilogyaxis}[xmin=1,ymax=1,width=7cm, clip marker paths=true, axis on top=true, enlarge y limits,
    legend style={at={(0.97,0.9)},cells={anchor=west}},
    xlabel={Iteration number},
    ylabel={Unpreconditioned relative residual},
    extra y ticks = {\myextratick},
    extra y tick style = {grid=major},
    extra y tick labels={},
yticklabel={
\pgfmathparse{int(\tick)}
\ifnum\pgfmathresult=0
$1$
\else
\pgfmathparse{int(\tick/ln(10))}{$10^{\pgfmathresult}$}
\fi
},
            mark repeat=1,legend columns=1
]
\addplot[color=RdGy-B, very thick,mark options={solid,fill=.!75!white},mark=*, dashed] table [x=it, y index=4] {\loadedconvergence};\label{plot:geneo}
\addplot[color=RdGy-E, very thick,mark options={solid,fill=.!75!white},mark=square*, dashed] table [x=it, y index=5] {\loadedconvergence};\label{plot:svd}
\addplot[color=RdGy-G, very thick,mark options={solid,fill=.!75!white},mark=triangle*, dashed] table [x=it, y index=6] {\loadedconvergence};\label{plot:eps}
    \legend{GenEO, SVD, GEVP}
\end{semilogyaxis}
\end{tikzpicture} 
    \end{minipage}
    \begin{minipage}{0.35\textwidth}
\pgfplotstabletypeset[every head row/.style={before row=\hline,after row=\hline},
                  every last row/.style={after row=\hline},
                  every even row/.style={before row={\rowcolor[gray]{0.9}}},
                  columns/it/.style={column name=Iterations},
                  columns/Nc/.style={int detect,column name=$n_C$,dec sep align},
                  columns/gc/.style={column name=GC,dec sep align,fixed,precision=3},
                  columns/oc/.style={column name=OC,dec sep align,fixed,precision=3},
                  columns={it,Nc,gc,oc},
                  font={\footnotesize},
                  ]\sizetable
    \\[1cm]
\pgfplotstabletypeset[every head row/.style={before row=\hline,after row=\hline},
                  every last row/.style={after row=\hline},
                  every even row/.style={before row={\rowcolor[gray]{0.9}}},
                  columns/pc/.style={string type, column name=Preconditioner},
                  columns/ovl/.style={int detect,column name=Overlap},
                  columns/threshold/.style={string type, column name=Threshold},
                  columns={pc,ovl,threshold},
                  font={\footnotesize},
                  ]\paramtable
    \end{minipage}
    \caption{Bilaplace 2D (\pgfmathprintnumber{1002001} unknowns)}\label{fig:bilaplace}
\end{figure}

\pgfplotstableread{figures/diffusion_heterogeneous_3d.dat}\loadedconvergence
\pgfplotstableread{figures/diffusion_heterogeneous_3d_op.dat}\sizetable
\pgfplotstableread{figures/diffusion_heterogeneous_3d_param.dat}\paramtable
\begin{figure}
\pgfplotstablegetelem{0}{svd}\of{\loadedconvergence}
\pgfkeys{/pgf/fpu}
    \pgfplotstablecreatecol[create col/expr={\thisrow{geneo}/\pgfplotsretval}]{geneo-rel}\loadedconvergence
    \pgfplotstablecreatecol[create col/expr={\thisrow{svd}/\pgfplotsretval}]{svd-rel}\loadedconvergence
    \pgfplotstablecreatecol[create col/expr={\thisrow{eps}/\pgfplotsretval}]{eps-rel}\loadedconvergence
    \pgfmathparse{multiply(1e-8,1)}
    \let\myextratick=\pgfmathresult
    \pgfkeys{/pgf/fpu=false}
    \begin{minipage}{0.6\textwidth}
\begin{tikzpicture}
\begin{semilogyaxis}[xmin=1,ymax=1,width=7cm, clip marker paths=true, axis on top=true, enlarge y limits,
    legend style={at={(0.97,0.9)},cells={anchor=west}},
    xlabel={Iteration number},
    ylabel={Unpreconditioned relative residual},
    extra y ticks = {\myextratick},
    extra y tick style = {grid=major},
    extra y tick labels={},
yticklabel={
\pgfmathparse{int(\tick)}
\ifnum\pgfmathresult=0
$1$
\else
\pgfmathparse{int(\tick/ln(10))}{$10^{\pgfmathresult}$}
\fi
},
            mark repeat=1,legend columns=1
]
\addplot[color=RdGy-B, very thick,mark options={solid,fill=.!75!white},mark=*, dashed] table [x=it, y index=4] {\loadedconvergence};
\addplot[color=RdGy-E, very thick,mark options={solid,fill=.!75!white},mark=square*, dashed] table [x=it, y index=5] {\loadedconvergence};
\addplot[color=RdGy-G, very thick,mark options={solid,fill=.!75!white},mark=triangle*, dashed] table [x=it, y index=6] {\loadedconvergence};
    \legend{GenEO, SVD, GEVP}
\end{semilogyaxis}
\end{tikzpicture} 
    \end{minipage}
    \begin{minipage}{0.35\textwidth}
\pgfplotstabletypeset[every head row/.style={before row=\hline,after row=\hline},
                  every last row/.style={after row=\hline},
                  every even row/.style={before row={\rowcolor[gray]{0.9}}},
                  columns/it/.style={column name=Iterations},
                  columns/Nc/.style={int detect,column name=$n_C$,dec sep align},
                  columns/gc/.style={column name=GC,dec sep align,fixed,precision=3},
                  columns/oc/.style={column name=OC,dec sep align,fixed,precision=3},
                  columns={it,Nc,gc,oc},
                  font={\footnotesize},
                  ]\sizetable
                  \\[1cm]
\pgfplotstabletypeset[every head row/.style={before row=\hline,after row=\hline},
                  every last row/.style={after row=\hline},
                  every even row/.style={before row={\rowcolor[gray]{0.9}}},
                  columns/pc/.style={string type, column name=Preconditioner},
                  columns/ovl/.style={int detect,column name=Overlap},
                  columns/threshold/.style={string type, column name=Threshold},
                  columns={pc,ovl,threshold},
                  font={\footnotesize},
                  ]\paramtable
    \end{minipage}
    \caption{Diffusion 3D (\pgfmathprintnumber{531441} unknowns)}\label{fig:diffusion-heterogeneous}
\end{figure}

\pgfplotstableread{figures/elasticity_2d.dat}\loadedconvergence
\pgfplotstableread{figures/elasticity_2d_op.dat}\sizetable
\pgfplotstableread{figures/elasticity_2d_param.dat}\paramtable
\begin{figure}
\pgfplotstablegetelem{0}{svd}\of{\loadedconvergence}
\pgfkeys{/pgf/fpu}
    \pgfplotstablecreatecol[create col/expr={\thisrow{geneo}/\pgfplotsretval}]{geneo-rel}\loadedconvergence
    \pgfplotstablecreatecol[create col/expr={\thisrow{svd}/\pgfplotsretval}]{svd-rel}\loadedconvergence
    \pgfplotstablecreatecol[create col/expr={\thisrow{eps}/\pgfplotsretval}]{eps-rel}\loadedconvergence
    \pgfmathparse{multiply(1e-8,1)}
    \let\myextratick=\pgfmathresult
    \pgfkeys{/pgf/fpu=false}
    \begin{minipage}{0.6\textwidth}
\begin{tikzpicture}
\begin{semilogyaxis}[xmin=1,ymax=1,width=7cm, clip marker paths=true, axis on top=true, enlarge y limits,
    legend style={at={(0.97,0.9)},cells={anchor=west}},
    xlabel={Iteration number},
    ylabel={Unpreconditioned relative residual},
    extra y ticks = {\myextratick},
    extra y tick style = {grid=major},
    extra y tick labels={},
yticklabel={
\pgfmathparse{int(\tick)}
\ifnum\pgfmathresult=0
$1$
\else
\pgfmathparse{int(\tick/ln(10))}{$10^{\pgfmathresult}$}
\fi
},
            mark repeat=1,legend columns=1
]
\addplot[color=RdGy-B, very thick,mark options={solid,fill=.!75!white},mark=*, dashed] table [x=it, y index=4] {\loadedconvergence};
\addplot[color=RdGy-E, very thick,mark options={solid,fill=.!75!white},mark=square*, dashed] table [x=it, y index=5] {\loadedconvergence};
\addplot[color=RdGy-G, very thick,mark options={solid,fill=.!75!white},mark=triangle*, dashed] table [x=it, y index=6] {\loadedconvergence};
    \legend{GenEO, SVD, GEVP}
\end{semilogyaxis}
\end{tikzpicture} 
    \end{minipage}
    \begin{minipage}{0.35\textwidth}
\pgfplotstabletypeset[every head row/.style={before row=\hline,after row=\hline},
                  every last row/.style={after row=\hline},
                  every even row/.style={before row={\rowcolor[gray]{0.9}}},
                  columns/it/.style={column name=Iterations},
                  columns/Nc/.style={int detect,column name=$n_C$,dec sep align},
                  columns/gc/.style={column name=GC,dec sep align,fixed,precision=3},
                  columns/oc/.style={column name=OC,dec sep align,fixed,precision=3},
                  columns={it,Nc,gc,oc},
                  font={\footnotesize},
                  ]\sizetable
                  \\[1cm]
\pgfplotstabletypeset[every head row/.style={before row=\hline,after row=\hline},
                  every last row/.style={after row=\hline},
                  every even row/.style={before row={\rowcolor[gray]{0.9}}},
                  columns/pc/.style={string type, column name=Preconditioner},
                  columns/ovl/.style={int detect,column name=Overlap},
                  columns/threshold/.style={string type, column name=Threshold},
                  columns={pc,ovl,threshold},
                  font={\footnotesize},
                  ]\paramtable
    \end{minipage}
    \caption{Elasticity 2D (\pgfmathprintnumber{1084202} unknowns)}\label{fig:elasticity-2d}
\end{figure}

\pgfplotstableread{figures/elasticity_3d.dat}\loadedconvergence
\pgfplotstableread{figures/elasticity_3d_op.dat}\sizetable
\pgfplotstableread{figures/elasticity_3d_param.dat}\paramtable
\begin{figure}
\pgfplotstablegetelem{0}{svd}\of{\loadedconvergence}
\pgfkeys{/pgf/fpu}
    \pgfplotstablecreatecol[create col/expr={\thisrow{geneo}/\pgfplotsretval}]{geneo-rel}\loadedconvergence
    \pgfplotstablecreatecol[create col/expr={\thisrow{svd}/\pgfplotsretval}]{svd-rel}\loadedconvergence
    \pgfplotstablecreatecol[create col/expr={\thisrow{eps}/\pgfplotsretval}]{eps-rel}\loadedconvergence
    \pgfmathparse{multiply(1e-8,1)}
    \let\myextratick=\pgfmathresult
    \pgfkeys{/pgf/fpu=false}
    \begin{minipage}{0.6\textwidth}
\begin{tikzpicture}
\begin{semilogyaxis}[xmin=1,ymax=1,width=7cm, clip marker paths=true, axis on top=true, enlarge y limits,
    legend style={at={(0.97,0.9)},cells={anchor=west}},
    xlabel={Iteration number},
    ylabel={Unpreconditioned relative residual},
    extra y ticks = {\myextratick},
    extra y tick style = {grid=major},
    extra y tick labels={},
yticklabel={
\pgfmathparse{int(\tick)}
\ifnum\pgfmathresult=0
$1$
\else
\pgfmathparse{int(\tick/ln(10))}{$10^{\pgfmathresult}$}
\fi
},
            mark repeat=1,legend columns=1
]
\addplot[color=RdGy-B, very thick,mark options={solid,fill=.!75!white},mark=*, dashed] table [x=it, y index=4] {\loadedconvergence};
\addplot[color=RdGy-E, very thick,mark options={solid,fill=.!75!white},mark=square*, dashed] table [x=it, y index=5] {\loadedconvergence};
\addplot[color=RdGy-G, very thick,mark options={solid,fill=.!75!white},mark=triangle*, dashed] table [x=it, y index=6] {\loadedconvergence};
    \legend{GenEO, SVD, GEVP}
\end{semilogyaxis}
\end{tikzpicture} 
    \end{minipage}
    \begin{minipage}{0.35\textwidth}
\pgfplotstabletypeset[every head row/.style={before row=\hline,after row=\hline},
                  every last row/.style={after row=\hline},
                  every even row/.style={before row={\rowcolor[gray]{0.9}}},
                  columns/it/.style={column name=Iterations},
                  columns/Nc/.style={int detect,column name=$n_C$,dec sep align},
                  columns/gc/.style={column name=GC,dec sep align,fixed,precision=3},
                  columns/oc/.style={column name=OC,dec sep align,fixed,precision=3},
                  columns={it,Nc,gc,oc},
                  font={\footnotesize},
                  ]\sizetable
                  \\[1cm]
\pgfplotstabletypeset[every head row/.style={before row=\hline,after row=\hline},
                  every last row/.style={after row=\hline},
                  every even row/.style={before row={\rowcolor[gray]{0.9}}},
                  columns/pc/.style={string type, column name=Preconditioner},
                  columns/ovl/.style={int detect,column name=Overlap},
                  columns/threshold/.style={string type, column name=Threshold},
                  columns={pc,ovl,threshold},
                  font={\footnotesize},
                  ]\paramtable
    \end{minipage}
    \caption{Elasticity 3D (\pgfmathprintnumber{521823} unknowns)}\label{fig:elasticity-3d}
\end{figure}

\pgfplotstableread{figures/stokes_2d.dat}\loadedconvergence
\pgfplotstablegetelem{0}{[index]1}\of{\loadedconvergence}
\pgfplotstableread{figures/stokes_2d_op.dat}\sizetable
\pgfplotstableread{figures/stokes_2d_param.dat}\paramtable
\begin{figure}
\pgfkeys{/pgf/fpu}
    \pgfplotstablecreatecol[create col/expr={\thisrow{svd}/\pgfplotsretval}]{svd-rel}\loadedconvergence
    \pgfmathparse{multiply(1e-8,1)}
    \let\myextratick=\pgfmathresult
    \pgfkeys{/pgf/fpu=false}
    \begin{minipage}{0.6\textwidth}
\begin{tikzpicture}
\begin{semilogyaxis}[xmin=1,ymax=1,width=7cm, clip marker paths=true, axis on top=true, enlarge y limits,
    legend style={at={(0.97,0.9)},cells={anchor=west}},
    xlabel={Iteration number},
    ylabel={Unpreconditioned relative residual},
    extra y ticks = {\myextratick},
    extra y tick style = {grid=major},
    extra y tick labels={},
yticklabel={
\pgfmathparse{int(\tick)}
\ifnum\pgfmathresult=0
$1$
\else
\pgfmathparse{int(\tick/ln(10))}{$10^{\pgfmathresult}$}
\fi
},
            mark repeat=1,legend columns=1
]
    \addplot[color=RdGy-E, very thick,mark options={solid,fill=.!75!white},mark=square*, dashed] table [x=it, y index=2] {\loadedconvergence};
    \legend{SVD}
\end{semilogyaxis}
\end{tikzpicture} 
    \end{minipage}
    \begin{minipage}{0.35\textwidth}
\pgfplotstabletypeset[every head row/.style={before row=\hline,after row=\hline},
                  every last row/.style={after row=\hline},
                  every even row/.style={before row={\rowcolor[gray]{0.9}}},
                  columns/it/.style={column name=Iterations},
                  columns/Nc/.style={int detect,column name=$n_C$,dec sep align},
                  columns/gc/.style={column name=GC,dec sep align,fixed,precision=3},
                  columns/oc/.style={column name=OC,dec sep align,fixed,precision=3},
                  columns={it,Nc,gc,oc},
                  font={\footnotesize},
                  ]\sizetable
                  \\[1cm]
\pgfplotstabletypeset[every head row/.style={before row=\hline,after row=\hline},
                  every last row/.style={after row=\hline},
                  every even row/.style={before row={\rowcolor[gray]{0.9}}},
                  columns/pc/.style={string type, column name=Preconditioner},
                  columns/ovl/.style={int detect,column name=Overlap},
                  columns/threshold/.style={string type, column name=Threshold},
                  columns={pc,ovl,threshold},
                  font={\footnotesize},
                  ]\paramtable
    \end{minipage}
    \caption{Stokes 2D (\pgfmathprintnumber{1444003} unknowns)}\label{fig:stokes-2d}
\end{figure}

\pgfplotstableread{figures/stokes_3d.dat}\loadedconvergence
\pgfplotstableread{figures/stokes_3d_op.dat}\sizetable
\pgfplotstableread{figures/stokes_3d_param.dat}\paramtable
\begin{figure}
\pgfplotstablegetelem{0}{svd}\of{\loadedconvergence}
\pgfkeys{/pgf/fpu}
    \pgfplotstablecreatecol[create col/expr={\thisrow{svd}/\pgfplotsretval}]{svd-rel}\loadedconvergence
    \pgfmathparse{multiply(1e-8,1)}
    \let\myextratick=\pgfmathresult
    \pgfkeys{/pgf/fpu=false}
    \begin{minipage}{0.6\textwidth}
\begin{tikzpicture}
\begin{semilogyaxis}[xmin=1,ymax=1,width=7cm, clip marker paths=true, axis on top=true, enlarge y limits,
    legend style={at={(0.97,0.9)},cells={anchor=west}},
    xlabel={Iteration number},
    ylabel={Unpreconditioned relative residual},
    extra y ticks = {\myextratick},
    extra y tick style = {grid=major},
    extra y tick labels={},
yticklabel={
\pgfmathparse{int(\tick)}
\ifnum\pgfmathresult=0
$1$
\else
\pgfmathparse{int(\tick/ln(10))}{$10^{\pgfmathresult}$}
\fi
},
            mark repeat=1,legend columns=1
]
\addplot[color=RdGy-E, very thick,mark options={solid,fill=.!75!white},mark=square*, dashed] table [x=it, y index=2] {\loadedconvergence};
    \legend{SVD}
\end{semilogyaxis}
\end{tikzpicture} 
    \end{minipage}
    \begin{minipage}{0.35\textwidth}
\pgfplotstabletypeset[every head row/.style={before row=\hline,after row=\hline},
                  every last row/.style={after row=\hline},
                  every even row/.style={before row={\rowcolor[gray]{0.9}}},
                  columns/it/.style={column name=Iterations},
                  columns/Nc/.style={int detect,column name=$n_C$,dec sep align},
                  columns/gc/.style={column name=GC,dec sep align,fixed,precision=3},
                  columns/oc/.style={column name=OC,dec sep align,fixed,precision=3},
                  columns={it,Nc,gc,oc},
                  font={\footnotesize},
                  ]\sizetable
                  \\[1cm]
\pgfplotstabletypeset[every head row/.style={before row=\hline,after row=\hline},
                  every last row/.style={after row=\hline},
                  every even row/.style={before row={\rowcolor[gray]{0.9}}},
                  columns/pc/.style={string type, column name=Preconditioner},
                  columns/ovl/.style={int detect,column name=Overlap},
                  columns/threshold/.style={string type, column name=Threshold},
                  columns={pc,ovl,threshold},
                  font={\footnotesize},
                  ]\paramtable
    \end{minipage}
    \caption{Stokes 3D (\pgfmathprintnumber{1663244} unknowns)}\label{fig:stokes-3d}
\end{figure}

\pgfplotstableread{figures/advection_2d.dat}\loadedconvergence
\pgfplotstableread{figures/advection_2d_op.dat}\sizetable
\pgfplotstableread{figures/advection_2d_param.dat}\paramtable
\begin{figure}
\pgfplotstablegetelem{0}{svd}\of{\loadedconvergence}
\pgfkeys{/pgf/fpu}
    \pgfplotstablecreatecol[create col/expr={\thisrow{svd}/\pgfplotsretval}]{svd-rel}\loadedconvergence
    \pgfplotstablecreatecol[create col/expr={\thisrow{block}/\pgfplotsretval}]{block-rel}\loadedconvergence
    \pgfmathparse{multiply(1e-8,1)}
    \let\myextratick=\pgfmathresult
    \pgfkeys{/pgf/fpu=false}
    \begin{minipage}{0.6\textwidth}
\begin{tikzpicture}
\begin{semilogyaxis}[xmin=1,ymax=1,width=7cm, clip marker paths=true, axis on top=true, enlarge y limits,
    legend style={at={(0.97,0.9)},cells={anchor=west}},
    xlabel={Iteration number},
    ylabel={Unpreconditioned relative residual},
    extra y ticks = {\myextratick},
    extra y tick style = {grid=major},
    extra y tick labels={},
yticklabel={
\pgfmathparse{int(\tick)}
\ifnum\pgfmathresult=0
$1$
\else
\pgfmathparse{int(\tick/ln(10))}{$10^{\pgfmathresult}$}
\fi
},
            mark repeat=1,legend columns=1
]
\addplot[color=RdGy-E, very thick,mark options={solid,fill=.!75!white},mark=square*, dashed] table [x=it, y index=4] {\loadedconvergence};
\addplot[color=RdGy-K, very thick,mark options={solid,fill=.!75!white},mark=diamond*, dashed] table [x=it, y index=3] {\loadedconvergence};\label{plot:block}
    \legend{SVD, Block splitting}
\end{semilogyaxis}
\end{tikzpicture} 
    \end{minipage}
    \begin{minipage}{0.35\textwidth}
\pgfplotstabletypeset[every head row/.style={before row=\hline,after row=\hline},
                  every last row/.style={after row=\hline},
                  every even row/.style={before row={\rowcolor[gray]{0.9}}},
                  columns/it/.style={column name=Iterations},
                  columns/Nc/.style={int detect,column name=$n_C$,dec sep align},
                  columns/gc/.style={column name=GC,dec sep align,fixed,precision=3},
                  columns/oc/.style={column name=OC,dec sep align,fixed,precision=3},
                  columns={it,Nc,gc,oc},
                  font={\footnotesize},
                  ]\sizetable
                  \\[1cm]
\pgfplotstabletypeset[every head row/.style={before row=\hline,after row=\hline},
                  every last row/.style={after row=\hline},
                  every even row/.style={before row={\rowcolor[gray]{0.9}}},
                  columns/pc/.style={string type, column name=Preconditioner},
                  columns/ovl/.style={int detect,column name=Overlap},
                  columns/threshold/.style={string type, column name=Threshold},
                  columns={pc,ovl,threshold},
                  font={\footnotesize},
                  ]\paramtable
    \end{minipage}
    \caption{Advection 2D (\pgfmathprintnumber{2253001} unknowns)}\label{fig:advection-2d}
\end{figure}

\pgfplotstableread{figures/advection_3d.dat}\loadedconvergence
\pgfplotstableread{figures/advection_3d_op.dat}\sizetable
\pgfplotstableread{figures/advection_3d_param.dat}\paramtable
\begin{figure}
\pgfplotstablegetelem{0}{svd}\of{\loadedconvergence}
\pgfkeys{/pgf/fpu}
    \pgfplotstablecreatecol[create col/expr={\thisrow{svd}/\pgfplotsretval}]{svd-rel}\loadedconvergence
    \pgfplotstablecreatecol[create col/expr={\thisrow{block}/\pgfplotsretval}]{block-rel}\loadedconvergence
    \pgfmathparse{multiply(1e-8,1)}
    \let\myextratick=\pgfmathresult
    \pgfkeys{/pgf/fpu=false}
    \begin{minipage}{0.6\textwidth}
\begin{tikzpicture}
\begin{semilogyaxis}[xmin=1,ymax=1,width=7cm, clip marker paths=true, axis on top=true, enlarge y limits,
    legend style={at={(0.97,0.9)},cells={anchor=west}},
    xlabel={Iteration number},
    ylabel={Unpreconditioned relative residual},
    extra y ticks = {\myextratick},
    extra y tick style = {grid=major},
    extra y tick labels={},
yticklabel={
\pgfmathparse{int(\tick)}
\ifnum\pgfmathresult=0
$1$
\else
\pgfmathparse{int(\tick/ln(10))}{$10^{\pgfmathresult}$}
\fi
},
            mark repeat=1,legend columns=1
]
\addplot[color=RdGy-E, very thick,mark options={solid,fill=.!75!white},mark=square*, dashed] table [x=it, y index=4] {\loadedconvergence};
\addplot[color=RdGy-K, very thick,mark options={solid,fill=.!75!white},mark=diamond*, dashed] table [x=it, y index=3] {\loadedconvergence};
    \legend{SVD, Block splitting}
\end{semilogyaxis}
\end{tikzpicture} 
    \end{minipage}
    \begin{minipage}{0.35\textwidth}
\pgfplotstabletypeset[every head row/.style={before row=\hline,after row=\hline},
                  every last row/.style={after row=\hline},
                  every even row/.style={before row={\rowcolor[gray]{0.9}}},
                  columns/it/.style={column name=Iterations},
                  columns/Nc/.style={int detect,column name=$n_C$,dec sep align},
                  columns/gc/.style={column name=GC,dec sep align,fixed,precision=3},
                  columns/oc/.style={column name=OC,dec sep align,fixed,precision=3},
                  columns={it,Nc,gc,oc},
                  font={\footnotesize},
                  ]\sizetable
                  \\[1cm]
\pgfplotstabletypeset[every head row/.style={before row=\hline,after row=\hline},
                  every last row/.style={after row=\hline},
                  every even row/.style={before row={\rowcolor[gray]{0.9}}},
                  columns/pc/.style={string type, column name=Preconditioner},
                  columns/ovl/.style={int detect,column name=Overlap},
                  columns/threshold/.style={string type, column name=Threshold},
                  columns={pc,ovl,threshold},
                  font={\footnotesize},
                  ]\paramtable
    \end{minipage}
    \caption{Advection 3D (\pgfmathprintnumber{8120601} unknowns)}\label{fig:advection-3d}
\end{figure}

\Cref{tab:diffusion-3d} reports the results of a weak-scaling analysis. The mesh is iteratively refined uniformly while increasing the number of subdomains. The constant-coefficient three-dimensional Poisson equation discretized with piecewise-linear Lagrange finite elements is solved. The analysis starts with two subdomains and a problem of dimension \pgfmathprintnumber{29791}. It ends with 512 subdomains and a problem of dimension \pgfmathprintnumber{6967871}. While there is a slight increase in grid complexity, the number of iterates remains stable. This observation holds when the preconditioner is assembled using either local singular vectors (\cref{subfig:svd}) or local eigenvectors (\cref{subfig:eps}). Note that in this experiment, the relative tolerance is set to $10^{-10}$ (instead of $10^{-8}$ of the previous figures).
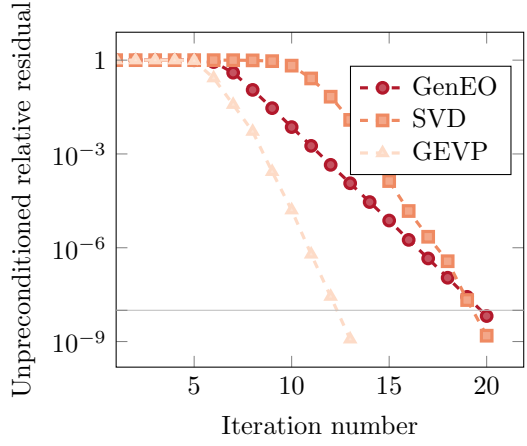
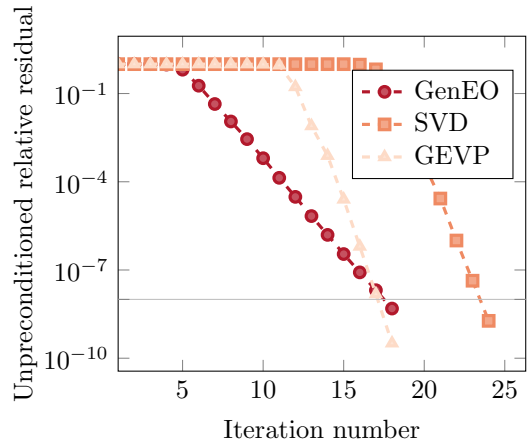
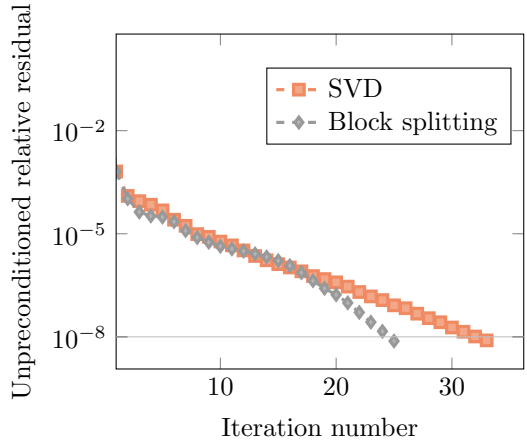
\begin{table}
\centering
    \subfloat[\ref{plot:svd}SVD\label{subfig:svd}]{
\pgfplotstableread{figures/scaling_svd.dat}\loadedconvergence
\pgfplotstabletypeset[every head row/.style={before row=\hline,after row=\hline},
                  every last row/.style={after row=\hline},
                  every even row/.style={before row={\rowcolor[gray]{0.9}}},
                  columns/it/.style={column name=Iterations},
                  columns/N/.style={int detect,column name=$N$,sort,dec sep align},
                  columns/nc/.style={int detect,column name=$n_C$,dec sep align},
                  columns/gc/.style={column name=GC,dec sep align,fixed,precision=3},
                  columns/oc/.style={column name=OC,dec sep align,fixed,precision=3},
                  columns={N,it,nc,gc,oc},
                  font={\footnotesize},
                  ]\loadedconvergence
    }
    \hfill
    \subfloat[\ref{plot:eps}GEVP\label{subfig:eps}]{
\pgfplotstableread{figures/scaling_eps.dat}\loadedconvergence
\pgfplotstabletypeset[every head row/.style={before row=\hline,after row=\hline},
                  every last row/.style={after row=\hline},
                  every even row/.style={before row={\rowcolor[gray]{0.9}}},
                  columns/it/.style={column name=Iterations},
                  columns/N/.style={int detect,column name=$N$,sort,dec sep align},
                  columns/nc/.style={int detect,column name=$n_C$,dec sep align},
                  columns/gc/.style={column name=GC,dec sep align,fixed,precision=3},
                  columns/oc/.style={column name=OC,dec sep align,fixed,precision=3},
                  columns={N,it,nc,gc,oc},
                  font={\footnotesize},
                  ]\loadedconvergence
}
    \caption{Diffusion 3D (approximately \pgfmathprintnumber{15000} interior unknowns per process)}\label{tab:diffusion-3d}
\end{table}

\section{Conclusion}
\label{sec:conclusion}
We presented in this paper new fully algebraic spectral coarse spaces for symmetric positive definite as well as nonsingular general matrices. The first spectral coarse space for SPD matrices requires the solution of one local, per subdomain, generalized eigenvalue problem. The bound on the condition number of the preconditioned matrix is independent of the number of subdomains. But, the dimension of the coarse space can be very large. Another local generalized eigenvalue problem is proposed to reduce the dimension of the coarse space. Therefore, two generalized eigenvalue problems are proposed to construct a two-level Schwarz preconditioner that provides an adaptive bound on the condition number of the preconditioned matrix. Numerical experiments showed the effectiveness of the proposed method and its competitiveness against state-of-the-art two-level Schwarz preconditioners. 
In the general case, a generalization of the local generalized eigenvalue problems is proposed to construct spectral coarse spaces for general nonsingular matrices without theoretical background, though. Nonetheless, numerical experiments demonstrated that the proposed method can be applied to a wide range of problems and prove effective for very challenging problems.
For SPD and general matrices, we found that there is a trade-off between increasing the number of layers in the overlap or fixing it and decreasing the truncation threshold $\tau$.
\appendix
\section*{Acknowledgments}
This work was granted access to the
GENCI-sponsored HPC resources of TGCC@CEA under allocation AD010607519R1. 

\bibliographystyle{siamplain}
\bibliography{main}
\end{document}